\documentclass[10pt]{amsart}
\usepackage{amsmath}
\usepackage{amsxtra}
\usepackage{amscd}
\usepackage{amsthm}
\usepackage{amsfonts}
\usepackage{amssymb}
\usepackage{eucal}
\usepackage[all]{xy}
\usepackage{graphicx}

\newtheorem{cor}[subsubsection]{Corollary}
\newtheorem{lem}[subsubsection]{Lemma}
\newtheorem{prop}[subsubsection]{Proposition}

\newtheorem{conj}[subsubsection]{Conjecture}
\newtheorem{thm}[subsubsection]{Theorem}

\theoremstyle{definition}

\newcommand{\thmref}[1]{Theorem~\ref{#1}}
\newcommand{\secref}[1]{Sect.~\ref{#1}}
\newcommand{\lemref}[1]{Lemma~\ref{#1}}
\newcommand{\propref}[1]{Proposition~\ref{#1}}
\newcommand{\corref}[1]{Corollary~\ref{#1}}
\newcommand{\conjref}[1]{Conjecture~\ref{#1}}

\theoremstyle{remark}
\newtheorem{rem}[subsubsection]{Remark}

\numberwithin{equation}{section}

\newcommand{\nc}{\newcommand}
\nc{\renc}{\renewcommand}
\nc{\ssec}{\subsection}
\nc{\sssec}{\subsubsection}
\nc{\on}{\operatorname}

\nc\ol{\overline}
\nc\wt{\widetilde}
\nc\tboxtimes{\wt{\boxtimes}}
\nc{\alp}{\alpha}

\nc{\ZZ}{{\mathbb Z}}
\nc{\NN}{{\mathbb N}}
\nc{\OO}{{\mathbb O}}
\renc{\SS}{{\mathbb S}}
\nc{\DD}{{\mathbb D}}
\nc{\GG}{{\mathbb G}}

\nc{\Fq}{{\mathbb F}_q}
\nc{\Fqb}{\ol{{\mathbb F}_q}}
\nc{\Ql}{\ol{{\mathbb Q}_\ell}}
\nc{\id}{\text{id}}
\nc\X{\mathcal X}

\nc{\Hom}{\on{Hom}}
\nc{\Lie}{\on{Lie}}
\nc{\Loc}{\on{Loc}}
\nc{\Pic}{\on{Pic}}
\nc{\Bun}{\on{Bun}}
\nc{\IC}{\on{IC}}
\nc{\Aut}{\on{Aut}}
\nc{\rk}{\on{rk}}
\nc{\Sh}{\on{Sh}}
\nc{\Perv}{\on{Perv}}
\nc{\pos}{{\on{pos}}}
\nc{\Conv}{\on{Conv}}
\nc{\Sph}{\on{Sph}}
\nc{\Sym}{\on{Sym}}
\nc{\BunBb}{\overline{\Bun}_B}
\nc{\BunBbm}{\overline{\Bun}_{B^-}}
\nc{\BunBbel}{\overline{\Bun}_{B,el}}
\nc{\BunBbmel}{\overline{\Bun}_{B^-,el}}
\nc{\Buno}{\overset{o}{\Bun}}
\nc{\BunPb}{{\overline{\Bun}_P}}
\nc{\BunBM}{\Bun_{B(M)}}
\nc{\BunBMb}{\overline{\Bun}_{B(M)}}
\nc{\BunPbw}{{\widetilde{\Bun}_P}}
\nc{\BunBP}{\widetilde{\Bun}_{B,P}}
\nc{\GUb}{\overline{G/U}}
\nc{\GUPb}{\overline{G/U(P)}}

\nc{\Hhom}{\underline{\on{Hom}}}
\nc\syminfty{\on{Sym}^{\infty}}
\nc\lal{\ol{\lambda}}
\nc\xl{\ol{x}}
\nc\thl{\ol{\theta}}
\nc\nul{\ol{\nu}}
\nc\mul{\ol{\mu}}
\nc\Sum\Sigma
\nc{\oX}{\overset{o}{X}{}}
\nc{\hl}{\overset{\leftarrow}h{}}
\nc{\hr}{\overset{\rightarrow}h{}}
\nc{\M}{{\mathcal M}}
\nc{\N}{{\mathcal N}}
\nc{\F}{{\mathcal F}}
\nc{\D}{{\mathcal D}}
\nc{\Q}{{\mathcal Q}}
\nc{\Y}{{\mathcal Y}}
\nc{\G}{{\mathcal G}}
\nc{\E}{{\mathcal E}}
\nc{\CalC}{{\mathcal C}}
\nc\Dh{\widehat{\D}}

\nc{\C}{{\mathcal C}}
\nc{\K}{{\mathcal K}}
\renewcommand{\H}{{\mathcal H}}

\nc{\T}{{\mathcal T}}
\nc{\V}{{\mathcal V}}
\renc{\P}{{\mathcal P}}
\nc{\A}{{\mathcal A}}
\nc{\B}{{\mathcal B}}
\nc{\U}{{\mathcal U}}

\nc{\Gr}{{\on{Gr}}}

\nc{\frn}{{\check{\mathfrak u}(P)}}
\nc{\p}{\mathfrak p}
\nc{\q}{\mathfrak q}
\nc\f{{\mathfrak f}}

\nc{\qo}{{\mathfrak q}}
\nc{\po}{{\mathfrak p}}
\nc{\s}{{\mathfrak s}}
\nc\w{\text{w}}

\nc\mathi\iota
\nc\Spec{\on{Spec}}
\nc\Mod{\on{Mod}}
\nc{\tw}{\widetilde{\mathfrak t}}
\nc{\pw}{\widetilde{\mathfrak p}}
\nc{\qw}{\widetilde{\mathfrak q}}
\nc{\jw}{\widetilde j}

\nc{\grb}{\overline{\Gr}}
\nc{\I}{\mathcal I}

\nc{\lambdach}{{\check\lambda}}
\nc{\Lambdach}{{\check\Lambda}{}}
\nc{\much}{{\check\mu}}
\nc{\omegach}{{\check\omega}}
\nc{\nuch}{{\check\nu}}
\nc{\etach}{{\check\eta}}
\nc{\alphach}{{\check\alpha}}
\nc{\betach}{{\check\beta}}
\nc{\rhoch}{{\check\rho}}
\nc{\ch}{{\check h}}

\nc{\Hb}{\overline{\H}}

\nc{\BA}{{\mathbb{A}}}
\nc{\BC}{{\mathbb{C}}}
\nc{\BG}{{\mathbb{G}}}
\nc{\BL}{{\mathbb{L}}}
\nc{\BM}{{\mathbb{M}}}
\nc{\BD}{{\mathbb{D}}}
\nc{\BN}{{\mathbb{N}}}
\nc{\BP}{{\mathbb{P}}}
\nc{\BQ}{{\mathbb{Q}}}
\nc{\BR}{{\mathbb{R}}}
\nc{\BZ}{{\mathbb{Z}}}
\nc{\BS}{{\mathbb{S}}}

\nc{\CA}{{\mathcal{A}}}
\nc{\CB}{{\mathcal{B}}}

\nc{\CE}{{\mathcal{E}}}
\nc{\CF}{{\mathcal{F}}}

\nc{\CL}{{\mathcal{L}}}
\nc{\CC}{{\mathcal{C}}}
\nc{\CM}{{\mathcal{M}}}
\nc{\CN}{{\mathcal{N}}}
\nc{\CK}{{\mathcal{K}}}
\nc{\CO}{{\mathcal{O}}}
\nc{\CP}{{\mathcal{P}}}
\nc{\CQ}{{\mathcal{Q}}}
\nc{\CR}{{\mathcal{R}}}
\nc{\CS}{{\mathcal{S}}}
\nc{\CT}{{\mathcal{T}}}
\nc{\CU}{{\mathcal{U}}}
\nc{\CV}{{\mathcal{V}}}
\nc{\CW}{{\mathcal{W}}}
\nc{\CX}{{\mathcal{X}}}
\nc{\CY}{{\mathcal{Y}}}
\nc{\CZ}{{\mathcal{Z}}}
\nc{\CI}{{\mathcal{I}}}

\nc{\csM}{{\check{\mathcal A}}{}}
\nc{\oM}{{\overset{\circ}{\mathcal M}}{}}
\nc{\obM}{{\overset{\circ}{\mathbf M}}{}}
\nc{\oCA}{{\overset{\circ}{\mathcal A}}{}}
\nc{\obA}{{\overset{\circ}{\mathbf A}}{}}
\nc{\ooM}{{\overset{\circ}{M}}{}}
\nc{\osM}{{\overset{\circ}{\mathsf M}}{}}
\nc{\vM}{{\overset{\bullet}{\mathcal M}}{}}
\nc{\nM}{{\underset{\bullet}{\mathcal M}}{}}
\nc{\oD}{{\overset{\circ}{\mathcal D}}{}}
\nc{\obD}{{\overset{\circ}{\mathbf D}}{}}
\nc{\oA}{{\overset{\circ}{\mathbb A}}{}}
\nc{\op}{{\overset{\bullet}{\mathbf p}}{}}
\nc{\cp}{{\overset{\circ}{\mathbf p}}{}}
\nc{\oU}{{\overset{\bullet}{\mathcal U}}{}}
\nc{\oZ}{{\overset{\circ}{\mathcal Z}}{}}
\nc{\ofZ}{{\overset{\circ}{\mathfrak Z}}{}}
\nc{\oF}{{\overset{\circ}{\fF}}}

\nc{\fa}{{\mathfrak{a}}}
\nc{\fb}{{\mathfrak{b}}}
\nc{\fd}{{\mathfrak{d}}}
\nc{\fg}{{\mathfrak{g}}}
\nc{\fgl}{{\mathfrak{gl}}}
\nc{\fh}{{\mathfrak{h}}}
\nc{\fj}{{\mathfrak{j}}}
\nc{\fl}{{\mathfrak{l}}}
\nc{\fm}{{\mathfrak{m}}}
\nc{\fn}{{\mathfrak{n}}}
\nc{\fu}{{\mathfrak{u}}}
\nc{\fp}{{\mathfrak{p}}}
\nc{\fr}{{\mathfrak{r}}}
\nc{\fs}{{\mathfrak{s}}}
\nc{\ft}{{\mathfrak{t}}}
\nc{\fsl}{{\mathfrak{sl}}}
\nc{\hsl}{{\widehat{\mathfrak{sl}}}}
\nc{\hgl}{{\widehat{\mathfrak{gl}}}}
\nc{\hg}{{\widehat{\mathfrak{g}}}}
\nc{\chg}{{\widehat{\mathfrak{g}}}{}^\vee}
\nc{\hn}{{\widehat{\mathfrak{n}}}}
\nc{\chn}{{\widehat{\mathfrak{n}}}{}^\vee}

\nc{\fA}{{\mathfrak{A}}}
\nc{\fB}{{\mathfrak{B}}}
\nc{\fD}{{\mathfrak{D}}}
\nc{\fE}{{\mathfrak{E}}}
\nc{\fF}{{\mathfrak{F}}}
\nc{\fG}{{\mathfrak{G}}}
\nc{\fK}{{\mathfrak{K}}}
\nc{\fL}{{\mathfrak{L}}}
\nc{\fM}{{\mathfrak{M}}}
\nc{\fN}{{\mathfrak{N}}}
\nc{\fP}{{\mathfrak{P}}}
\nc{\fU}{{\mathfrak{U}}}
\nc{\fV}{{\mathfrak{V}}}
\nc{\fZ}{{\mathfrak{Z}}}

\nc{\bb}{{\mathbf{b}}}
\nc{\bc}{{\mathbf{c}}}
\nc{\bd}{{\mathbf{d}}}
\nc{\be}{{\mathbf{e}}}
\nc{\bbf}{{\mathbf{f}}}
\nc{\bj}{{\mathbf{j}}}
\nc{\bn}{{\mathbf{n}}}
\nc{\bp}{{\mathbf{p}}}
\nc{\bq}{{\mathbf{q}}}
\nc{\bu}{{\mathbf{u}}}
\nc{\bv}{{\mathbf{v}}}
\nc{\bx}{{\mathbf{x}}}
\nc{\bs}{{\mathbf{s}}}
\nc{\by}{{\mathbf{y}}}
\nc{\bw}{{\mathbf{w}}}
\nc{\bA}{{\mathbf{A}}}
\nc{\bK}{{\mathbf{K}}}
\nc{\bB}{{\mathbf{B}}}
\nc{\bC}{{\mathbf{C}}}
\nc{\bG}{{\mathbf{G}}}
\nc{\bD}{{\mathbf{D}}}
\nc{\bH}{{\mathbf{H}}}
\nc{\bM}{{\mathbf{M}}}
\nc{\bN}{{\mathbf{N}}}
\nc{\bV}{{\mathbf{V}}}
\nc{\bW}{{\mathbf{W}}}
\nc{\bX}{{\mathbf{X}}}
\nc{\bZ}{{\mathbf{Z}}}
\nc{\bS}{{\mathbf{S}}}

\nc{\sA}{{\mathsf{A}}}
\nc{\sB}{{\mathsf{B}}}
\nc{\sC}{{\mathsf{C}}}
\nc{\sD}{{\mathsf{D}}}
\nc{\sF}{{\mathsf{F}}}
\nc{\sG}{{\mathsf{G}}}
\nc{\sK}{{\mathsf{K}}}
\nc{\sM}{{\mathsf{M}}}
\nc{\sO}{{\mathsf{O}}}
\nc{\sW}{{\mathsf{W}}}
\nc{\sQ}{{\mathsf{Q}}}
\nc{\sP}{{\mathsf{P}}}
\nc{\sZ}{{\mathsf{Z}}}
\nc{\sfp}{{\mathsf{p}}}
\nc{\sfq}{{\mathsf{q}}}
\nc{\sr}{{\mathsf{r}}}
\nc{\sk}{{\mathsf{k}}}
\nc{\sg}{{\mathsf{g}}}
\nc{\sff}{{\mathsf{f}}}
\nc{\sfb}{{\mathsf{b}}}
\nc{\sfc}{{\mathsf{c}}}
\nc{\sd}{{\mathsf{d}}}

\nc{\BK}{{\bar{K}}}

\nc{\tA}{{\widetilde{\mathbf{A}}}}
\nc{\tB}{{\widetilde{\mathcal{B}}}}
\nc{\tg}{{\widetilde{\mathfrak{g}}}}
\nc{\tG}{{\widetilde{G}}}
\nc{\TM}{{\widetilde{\mathbb{M}}}{}}
\nc{\tO}{{\widetilde{\mathsf{O}}}{}}
\nc{\tU}{{\widetilde{\mathfrak{U}}}{}}
\nc{\TZ}{{\tilde{Z}}}
\nc{\tx}{{\tilde{x}}}
\nc{\tbv}{{\tilde{\bv}}}
\nc{\tfP}{{\widetilde{\mathfrak{P}}}{}}
\nc{\tz}{{\tilde{\zeta}}}
\nc{\tmu}{{\tilde{\mu}}}

\nc{\urho}{\underline{\rho}}
\nc{\uB}{\underline{B}}
\nc{\uC}{{\underline{\mathbb{C}}}}
\nc{\ui}{\underline{i}}
\nc{\uj}{\underline{j}}
\nc{\ofP}{{\overline{\mathfrak{P}}}}
\nc{\oB}{{\overline{\mathcal{B}}}}
\nc{\og}{{\overline{\mathfrak{g}}}}
\nc{\oI}{{\overline{I}}}

\nc{\eps}{\varepsilon}
\nc{\hrho}{{\hat{\rho}}}

\nc{\one}{{\mathbf{1}}}
\nc{\two}{{\mathbf{t}}}

\nc{\Rep}{{\mathop{\operatorname{\rm Rep}}}}
\nc{\Tot}{{\mathop{\operatorname{\rm Tot}}}}
\nc{\Ker}{{\mathop{\operatorname{\rm Ker}}}}
\nc{\Hilb}{{\mathop{\operatorname{\rm Hilb}}}}
\nc{\End}{{\mathop{\operatorname{\rm End}}}}
\nc{\Ext}{{\mathop{\operatorname{\rm Ext}}}}
\nc{\CHom}{{\mathop{\operatorname{{\mathcal{H}}\it om}}}}
\nc{\GL}{{\mathop{\operatorname{\rm GL}}}}
\nc{\gr}{{\mathop{\operatorname{\rm gr}}}}
\nc{\Id}{{\mathop{\operatorname{\rm Id}}}}
\nc{\de}{{\mathop{\operatorname{\rm def}}}}
\nc{\length}{{\mathop{\operatorname{\rm length}}}}
\nc{\supp}{{\mathop{\operatorname{\rm supp}}}}

\nc{\Cliff}{{\mathsf{Cliff}}}
\nc{\Fl}{\on{Fl}}
\nc{\Fib}{{\mathsf{Fib}}}
\nc{\Coh}{{\on{Coh}}}
\nc{\FCoh}{{\mathsf{FCoh}}}

\nc{\reg}{{\text{\rm reg}}}

\nc{\cplus}{{\mathbf{C}_+}}
\nc{\cminus}{{\mathbf{C}_-}}
\nc{\cthree}{{\mathbf{C}_*}}
\nc{\Qbar}{{\bar{Q}}}
\nc\Eis{\on{Eis}}
\nc\Eisb{\ol\Eis{}}
\nc\wh{\widehat}
\nc{\Def}{\on{Def_{\check{\fb}}(E)}}
\nc{\barZ}{\overline{Z}{}}
\nc{\barbarZ}{\overline{\barZ}{}}
\nc{\barpi}{\overline\pi}
\nc{\barbarpi}{\overline\barpi}
\nc{\barpip}{\overline\pi{}^+}
\nc{\barpim}{\overline\pi{}^-}

\nc{\fq}{\mathfrak q}

\nc{\fqb}{\ol{\fq}{}}
\nc{\fpb}{\ol{\fp}{}}

\nc{\hattimes}{\wh\otimes}

\nc{\bh}{{\bar{h}}}
\nc{\bOmega}{{\overline{\Omega(\check \fn)}}}

\nc{\seq}[1]{\stackrel{#1}{\sim}}

\nc{\cT}{{\check{T}}}
\nc{\cG}{{\check{G}}}
\nc{\cP}{{\check{P}}}
\nc{\cM}{{\check{M}}}
\nc{\cB}{{\check{B}}}

\nc{\ct}{{\check{\mathfrak t}}}
\nc{\cg}{{\check{\fg}}}
\nc{\cb}{{\check{\fb}}}
\nc{\cn}{{\check{\fn}}}

\nc{\cLambda}{{\check\Lambda}}

\nc{\cla}{{\check\lambda}}
\nc{\cmu}{{\check\mu}}
\nc{\cnu}{{\check\nu}}
\nc{\ceta}{{\check\eta}}

\nc{\Dmod}{\on{D-mod}}

\nc{\DefbE}{{\on{Def}_{\cB}(E_\cT)}}

\nc{\imathb}{{\ol{\imath}}}

\nc{\Vect}{\on{Vect}}

\nc{\psId}{\on{Ps-Id}}
\nc{\sotimes}{\overset{!}\otimes}
\nc{\LocSys}{\on{LocSys}}
\nc{\QCoh}{\on{QCoh}}
\nc{\IndCoh}{\on{IndCoh}}

\begin{document}

\title[A ``strange" functional equation for Eisenstein series]{A ``strange" functional equation for Eisenstein series and 
miraculous duality on the moduli stack of bundles}

\author{D.~Gaitsgory}

\date{\today}

\begin{abstract}
We show that the failure of the usual Verdier duality on $\Bun_G$ leads to a new duality
functor on the category of D-modules, and we study its relation to the operation of Eisenstein
series.
\end{abstract}

\maketitle

\section*{Introduction}

\ssec{Context for the present work}

\sssec{}

This paper arose in the process of developing what V.~Drinfeld calls 
\emph{the geometric theory of automorphic functions}. I.e., we study \emph{sheaves} 
on the moduli stack $\Bun_G$ of principal $G$-bundles on a curve $X$. Here and elsewhere 
in the paper, we fix an algebraically closed ground field $k$, and we let $G$ be a reductive group and $X$
a smooth and complete curve over $k$.  

\medskip

In the bulk of the paper we will take $k$ to be of characteristic $0$, and by a ``sheaf" we will understand
an object of the derived category of D-modules. However, with appropriate modifications, 
our results apply also to $\ell$-adic sheaves, or any other reasonable sheaf-theoretic situation. 

\medskip

Much of the motivation for the study of sheaves on $\Bun_G$ comes from the so-called \emph{geometric Langlands program}. 
In line with this, 
the main results of this paper have a transparent meaning in terms of this program, see \secref{ss:geom lang}. 
However, one can also view them from the perspective of the classical theory of automorphic functions 
(rather, we will see phenomena that so far have not been studied classically). 

\sssec{Constant term and Eisenstein series functors}

To explain what is done in this paper we will first recall the main result of \cite{DrGa3}. 

\medskip

Let $P\subset G$ be a parabolic
subgroup with Levi quotient $M$. The diagram of groups 
$$G\hookleftarrow P\twoheadrightarrow M$$
gives rise to a diagram of stacks
\begin{equation} \label{e:basic diagram intro intro}
\xy
(-15,0)*+{\Bun_G}="X";
(15,0)*+{\Bun_M.}="Y";
(0,15)*+{\Bun_P}="Z";
{\ar@{->}_{\sfp} "Z";"X"};
{\ar@{->}^{\sfq} "Z";"Y"};
\endxy
\end{equation}

Using this diagram as ``pull-push", one can write down several functors connecting the categories
of D-modules on $\Bun_G$ and $\Bun_M$, respectively. By analogy with the classical theory of
automorphic functions, we call the functors going from $\Bun_M$ to $\Bun_G$ ``Eisenstein series",
and the functors going from $\Bun_G$ to $\Bun_M$ ``constant term".  

\medskip

Namely, we have
$$\Eis_!:=\sfp_!\circ \sfq^*, \quad \Dmod(\Bun_M)\to \Dmod(\Bun_G),$$
$$\Eis_*:=\sfp_*\circ \sfq^!, \quad \Dmod(\Bun_M)\to \Dmod(\Bun_G),$$
$$\on{CT}_!:=\sfq_!\circ \sfp^*, \quad \Dmod(\Bun_G)\to \Dmod(\Bun_M),$$
$$\on{CT}_*:=\sfq_*\circ \sfp^!, \quad \Dmod(\Bun_G)\to \Dmod(\Bun_M).$$  

Note that unlike the classical theory, where there is only one pull-back and one push-forward
for functions, for sheaves there are two options: $!$ and $*$, for both pull-back and push-forward. 
The interaction of these two options is one way to look at what this paper is about.

\medskip

Among the above functors, there are some obvious adjoint pairs: $\Eis_!$ is the left adjoint of
$\on{CT}_*$, and $\on{CT}_!$ is the left adjoint of $\Eis_*$. 

\medskip

In addition to this, the following, perhaps a little unexpected, result was proved in \cite{DrGa3}: 

\begin{thm} \label{t:CT}
The functors $\on{CT}_!$ and $\on{CT}^-_*$ are canonically isomorphic.
\end{thm}

In the statement of the theorem the superscript ``$-$" means 
the constant term functor taken with respect to the \emph{opposite} parabolic $P^-$  (note that
the Levi quotients of $P$ and $P^-$ are canonically identified). 

\medskip

Our goal in the present paper is to understand what implication the above-mentioned
isomorphism 
$$\on{CT}_!\simeq \on{CT}^-_*$$
has for the Eisenstein series functors $\Eis_!$ and $\Eis_*$.  The conclusion will be
what we will call a ``strange" functional equation \eqref{e:strange intro}, explained below. 

\medskip

In order to explain what the ``strange" functional equation does, we will need to go a little
deeper into what one may call the ``functional-analytic" aspects of the study of $\Bun_G$.

\sssec{Verdier duality on stacks}  \label{sss:dual first}

The starting point for the ``analytic" issues that we will be dealing with is that the stack $\Bun_G$
is \emph{not quasi-compact} (this is parallel to the fact that in the classical theory, the automorphic space is 
not compact, leading to a host of interesting analytic phenomena). The particular phenomenon
that we will focus on is the absence of the usual Verdier duality functor, and what replaces it. 

\medskip

First off, it is well-known (see, e.g., \cite[Sect. 2]{DrGa2}) that if $\CY$ is an arbitrary 
reasonable\footnote{The word ``reasonable" here does not have a technical meaning; the technical term
is ``QCA", which means that the automorphism group of any field-valued point is affine.} 
quasi-compact algebraic stack, then the category $\Dmod(\CY)$ is compactly generated and 
naturally self-dual.  

\medskip

Perhaps, the shortest way to understand the meaning of self-duality is that the subcategory
$\Dmod(\CY)^c\subset \Dmod(\CY)$ consisting of compact objects carries a canonically defined
contravariant self-equivalence, called Verdier duality. A more flexible 
way of interpreting the same phenomenon is an equivalence, denoted $\bD_\CY$, between $\Dmod(\CY)$
and its \emph{dual} category $\Dmod(\CY)^\vee$ (we refer the reader to \cite[Sect. 1]{DrGa1}, where the basics
of the notion of duality for DG categories are reviewed). 

\medskip

Let us now remove the assumption that $\CY$ be quasi-compact. Then there is another geometric condition,
called ``truncatability" that ensures that $\Dmod(\CY)$ is compactly generated (see \cite[Definition 4.1.1]{DrGa2},
where this notion is introduced). We remark here that the goal of the paper \cite{DrGa2} was to show
that the stack $\Bun_G$ is truncatable. The reader who is not familiar with this notion is advised to ignore
it on the first pass. 

\medskip

Thus, let us assume that $\CY$ is truncatable. However, there still is no obvious replacement for 
Verdier duality: extending the quasi-compact case, one can define a functor 
$$(\Dmod(\CY)^c)^{\on{op}}\to \Dmod(\CY),$$
but it no longer lands in $\Dmod(\CY)^c$ (unless $\CY$ is a disjoint union of quasi-compact stacks).
In the language of dual categories, we have a functor
$$\on{Ps-Id}_{\CY,\on{naive}}:\Dmod(\CY)^\vee\to \Dmod(\CY),$$
but it is no longer an equivalence.\footnote{The category $\Dmod(\CY)^\vee$ 
and the functor $\on{Ps-Id}_{\CY,\on{naive}}$ will be described explicitly in \secref{ss:dual}.}

\medskip
 
In particular, the functor $\on{Ps-Id}_{\Bun_G,\on{naive}}$ is \emph{not} an equivalence, unless $G$ is a torus.

\sssec{The pseudo-identity functor}  \label{sss:pseudo-id intro}

To potentially remedy this, V. ~Drinfeld suggested another functor, denoted
$$\on{Ps-Id}_{\CY,!}: \Dmod(\CY)^\vee\to \Dmod(\CY),$$
see \cite[Sect. 4.4.8]{DrGa2} or \secref{ss:pseudo} of the present paper. 

\medskip

Now, it is not true that for all truncatable stacks $\CY$, the functor $\on{Ps-Id}_{\CY,!}$ is an equivalence. 
In \cite{DrGa2} the stacks for which it is an equivalence are called ``miraculous". 

\medskip

We can now formulate the main result of this paper (conjectured by V.~Drinfeld): 

\begin{thm} \label{t:miraculous preview}
The stack $\Bun_G$ is miraculous.
\end{thm}

We repeat that the above theorem says that the canonically defined functor $\on{Ps-Id}_{\Bun_G,!}$
defines an identification of $\Dmod(\Bun_G)$ and its dual category. Equivalently, it gives rise to
a (non-obvious!) contravariant self-equivalence on $\Dmod(\Bun_G)^c$. 

\sssec{The ``strange" functional equation}

Finally, we can go back and state the ``strange" functional equation, which is in fact an ingredient
in the proof of \thmref{t:miraculous preview}:

\begin{thm} \label{t:strange preview}
We have a canonical isomorphism of functors 
$$\Eis^-_!\circ \on{Ps-Id}_{\Bun_M,!}\simeq \on{Ps-Id}_{\Bun_G,!}\circ (\on{CT}_*)^\vee.$$
\end{thm}

In the \thmref{t:strange preview}, the functor $(\on{CT}_*)^\vee$ maps
$$\Dmod(\Bun_M)^\vee\to \Dmod(\Bun_G)^\vee$$
and is the \emph{dual} of the functor $\on{CT}_*$. As we shall see in \secref{ss:dual of CT},
the functor $(\on{CT}_*)^\vee$ is a close relative of the functor $\Eis_*$, introduced
earlier. 

\ssec{Motivation from geometric Langlands}  \label{ss:geom lang}

We shall now proceed and describe how the results of this paper fit into the
geometric Langlands program. The contents of this subsection play a motivational role only, 
and the reader not familiar with the objects
discussed below can skip this subsection and proceed to \secref{ss:usual}. 

\sssec{Statement of GLC}

Let us recall the statement of the categorical geometric Langlands conjecture (GLC), according to \cite[Conjecture 10.2.2]{AG}.

\medskip

The left-hand (i.e., geometric) side of GLC is the DG category $\Dmod(\Bun_G)$ of D-modules on the stack $\Bun_G$. 

\medskip

Let $\cG$ denote the Langlands dual group of $G$, and let $\LocSys_\cG$ denote the (derived) stack
of $\cG$-local systems on $X$. The right-hand (i.e., spectral) side of GLC has to do with (quasi)-coherent sheaves on $\LocSys_\cG$.

\medskip

More precisely, In \cite{AG}, a certain modification of the DG category $\QCoh(\LocSys_\cG)$ was introduced; we denote it by
$\IndCoh_{\on{Nilp}_{glob}}(\LocSys_\cG)$.  This category is what appears on the the spectral side of GLC. 

\medskip

Thus, GLC states the existence of an equivalence
\begin{equation} \label{e:GLC}
\BL_G:\Dmod(\Bun_G)\to \IndCoh_{\on{Nilp}_{glob}}(\LocSys_\cG),
\end{equation}
that satisfies a number of properties that (conjecturally) determine $\BL_G$ uniquely. 

\medskip

The property of $\BL_G$, relevant for this paper, is the compatibility of \eqref{e:GLC} with the functor of Eisenstein series, 
see \secref{sss:GLC and Eis} below. 

\sssec{Interaction of GLC with duality}

A feature of the spectral side crucial for this paper is that the Serre duality functor of \cite[Proposition 3.7.2]{AG}
gives rise to an equivalance:
$$\bD^{\on{Serre}}_{\LocSys_\cG}:(\IndCoh_{\on{Nilp}_{glob}}(\LocSys_\cG))^\vee\to
\IndCoh_{\on{Nilp}_{glob}}(\LocSys_\cG).$$
(Here, as in \secref{sss:dual first}, for a compactly generated category $\bC$, we denote by $\bC^\vee$ the dual category.)

\medskip

Hence, if we believe in the existence of an equivalence $\BL_G$ of \eqref{e:GLC}, there should exist an equivalence
\begin{equation} \label{e:self-dual 1}
(\Dmod(\Bun_G))^\vee\simeq \Dmod(\Bun_G).
\end{equation}

Now, the pseudo-identity functor $\on{Ps-Id}_{\Bun_G,!}$ mentioned in \secref{sss:pseudo-id intro} 
and appearing in \thmref{t:miraculous preview} is exactly supposed to perform this role. 
More precisely, we can enhance the statement of GLC by specifying how it is supposed to interact
with duality: 

\begin{conj}  \label{c:Langlands+duality}
The diagram 
\begin{equation} \label{e:Langlands+duality}
\CD
\Dmod(\Bun_G)^\vee  @>{((\BL_G)^\vee)^{-1}}>>   (\IndCoh_{\on{Nilp}_{glob}}(\LocSys_\cG))^\vee  \\
& & @VV{\bD^{\on{Serre}}_{\LocSys_\cG}}V   \\
@V{\psId_{\Bun_G,!}}VV   \IndCoh_{\on{Nilp}_{glob}}(\LocSys_\cG) \\
& &     @VV{\tau}V   \\
\Dmod(\Bun_G)   @>{\BL_G}>> \IndCoh_{\on{Nilp}_{glob}}(\LocSys_\cG)
\endCD
\end{equation} 
commutes up to a cohomological shift, where
$\tau$ denotes the automorphism, induced by the Cartan involution 
of $G$.
\end{conj} 

\begin{rem}
Let us comment on the presence of the Cartan involution in \conjref{c:Langlands+duality}.
In fact, it can be seen already when $G$ is a torus $T$, in which case $\tau$ is the 
inversion automorphism. 

\medskip

Indeed, we let $\BL_T$ be the Fourier-Mukai equivalence, and 
\conjref{c:Langlands+duality} is known to hold.

\end{rem} 

\sssec{Interaction of GLC with Eisenstein series}   \label{sss:GLC and Eis}

Let us recall (following \cite[Conjecture 12.2.9]{AG} or \cite[Sect. 6.4.5]{Ga1}) how the equivalence $\BL_G$ is supposed to be 
compatible with the functor(s) of Eisenstein series. 
 
\medskip

For a (standard) parabolic $P\subset G$, let $\cP$ be the corresponding parabolic in $\cG$. Consider
the diagram
\begin{equation} \label{e:basic diagram spec}
\xy
(-15,0)*+{\LocSys_\cG}="X";
(15,0)*+{\LocSys_\cM.}="Y";
(0,15)*+{\LocSys_\cP}="Z";
{\ar@{->}_{\sfp_{\on{spec}}} "Z";"X"};
{\ar@{->}^{\sfq_{\on{spec}}} "Z";"Y"};
\endxy
\end{equation}

We define the functors of spectral Eisenstein series and constant term
$$\Eis_{\on{spec}}:\IndCoh(\LocSys_\cM)\to \IndCoh(\LocSys_\cG), \quad 
\Eis_{\on{spec}}:=(\sfp_{\on{spec}})_*\circ (\sfq_{\on{spec}})^*,$$
$$\on{CT}_{\on{spec}}:\IndCoh(\LocSys_\cG)\to \IndCoh(\LocSys_\cM), 
\quad \on{CT}_{\on{spec}}:=(\sfq_{\on{spec}})_*\circ (\sfp_{\on{spec}})^!.$$
see \cite[Sect. 12.2.1]{AG} for more details. 
The functors $(\Eis_{\on{spec}},\on{CT}_{\on{spec}})$ form an adjoint pair. 

\begin{rem} \label{r:Gore}
In \cite[Conjecture 12.2.9]{AG} a slightly different version of the functor $\Eis_{\on{spec}}$ is given, where instead
of the functor $(\sfq_{\on{spec}})^*$ we use $(\sfq_{\on{spec}})^!$. The difference between these two functors
is given by tensoring by a graded line bundle on $\LocSys_\cM$; this is due to the fact that the morphism $\sfq_{\on{spec}}$
is \emph{Gorenstein}. This difference will be immaterial for the purposes of this paper.
\end{rem}

\medskip

The compatibility of the geometric Langlands equivalence of \eqref{e:GLC}
with Eisenstein series reads (see \cite[Conjecture 12.2.9]{AG}):

\begin{conj} \label{c:Langlands+Eisenstein}
The diagram
\begin{equation} \label{e:Langlands+Eisenstein}
\CD
\Dmod(\Bun_G)   @>{\BL_G}>> \IndCoh_{\on{Nilp}_{glob}}(\LocSys_\cG) \\
@A{\Eis_!}AA    @AA{\Eis_{\on{spec}}}A   \\
\Dmod(\Bun_M)   @>{\BL_M}>> \IndCoh_{\on{Nilp}_{glob}}(\LocSys_\cM)
\endCD
\end{equation}
commutes up to an automorphism of $\IndCoh_{\on{Nilp}_{glob}}(\LocSys_\cM)$,
given by tensoring with a certain canonically defined graded line bundle on $\LocSys_\cM$.
\end{conj} 

\sssec{Recovering the ``strange" functional equation}

Let us now analyze what the combination of Conjectures \ref{c:Langlands+duality} and \ref{c:Langlands+Eisenstein}
says about the interaction of the functor $\psId_{\Bun_G,!}$ with $\Eis_!$. The conclusion that we will draw will
amount to \thmref{t:strange preview} of the present paper (the reader may safely choose to skip the derivation that
follows). 

\medskip

First, passing to the right adjoint and then dual functors in \eqref{e:Langlands+Eisenstein}, we obtain a diagram
\begin{equation} \label{e:Langlands+Eisenstein dual}
\CD
\Dmod(\Bun_G)^\vee   @>{(\BL_G^\vee)^{-1}}>> (\IndCoh_{\on{Nilp}_{glob}}(\LocSys_\cG)^\vee \\
@A{(\on{CT}_*)^\vee}AA    @AA{(\on{CT}_{\on{spec}})^\vee}A   \\
\Dmod(\Bun_M)^\vee   @>{(\BL_M^\vee)^{-1}}>> (\IndCoh_{\on{Nilp}_{glob}}(\LocSys_\cM))^\vee
\endCD
\end{equation}
that commutes up to a tensoring by a graded line bundle on $\LocSys_\cM$. 

\medskip

Next, we note that the diagram
\begin{equation} \label{e:coh duality and Eis}
\CD
(\IndCoh_{\on{Nilp}_{glob}}(\LocSys_\cG)^\vee   @>{\bD^{\on{Serre}}_{\LocSys_\cG}}>>   \IndCoh_{\on{Nilp}_{glob}}(\LocSys_\cG)  \\
@A{(\on{CT}_{\on{spec}})^\vee}AA   @A{\Eis_{\on{spec}}}AA \\ 
(\IndCoh_{\on{Nilp}_{glob}}(\LocSys_\cM)^\vee   @>{\bD^{\on{Serre}}_{\LocSys_\cM}}>>   \IndCoh_{\on{Nilp}_{glob}}(\LocSys_\cM)
\endCD
\end{equation} 
also commutes up to a tensoring by a graded line bundle on $\LocSys_\cM$, see Remark \ref{r:Gore}.

\medskip

Now, juxtaposing the diagrams \eqref{e:Langlands+Eisenstein}, \eqref{e:Langlands+Eisenstein dual},
\eqref{e:coh duality and Eis} with the diagrams \eqref{e:Langlands+duality} for the groups $G$ and $M$
respectively, we obtain a commutative diagram: 

\begin{equation} \label{e:strange intro}
\CD
\Dmod(\Bun_G)^\vee  @>{\psId_{\Bun_G,!}}>>    \Dmod(\Bun_G)  \\
@A{(\on{CT}_*)^\vee}AA    @AA{\tau_G\circ \Eis_!\circ \tau_M}A   \\
\Dmod(\Bun_M)^\vee  @>{\psId_{\Bun_M,!}}>>    \Dmod(\Bun_M).
\endCD
\end{equation}

Notice now that $\tau_G\circ \Eis_!\circ \tau_M\simeq \Eis^-_!$, so the commutative diagram  
\eqref{e:strange intro} recovers the isomorphism of \thmref{t:strange preview}. 

\ssec{The usual functional equation}  \label{ss:usual}

As was mentioned above, we view the commutativity of the diagram \eqref{e:strange intro} as a kind of 
``strange" functional equation, hence the title of this paper. 

\medskip

Let us now compare it to the usual functional equation of \cite[Theorem 2.1.8]{BG}. 

\sssec{}

In {\it loc.cit.} one considered
the case of $P=B$, the Borel subgroup and hence $M=T$, the abstract Cartan. We consider the full subcategory
$$\Dmod(\Bun_T)^{\on{reg}}\subset \Dmod(\Bun_T),$$
defined as in \cite[Sect. 2.1.7]{BG}. This is a full subcategory that under the Fourier-Mukai equivalence
$$\Dmod(\Bun_T)\simeq \QCoh(\LocSys_\cT)$$
corresponds to 
$$\QCoh(\LocSys_\cT^{\on{reg}})\hookrightarrow \QCoh(\LocSys_\cT),$$
where $\LocSys_\cT^{\on{reg}}\subset \LocSys_\cT$ is the open 
locus of $\LocSys_\cT$ consisting of those $\cT$-local systems that for every root $\alpha$
of $\cT$ induce a non-trivial local system for $\BG_m$. 

\medskip

Instead of the functor $\Eis_!$, or the functor that we introduce as $\Eis_*:=\sfp_*\circ \sfq^!$ (see \secref{sss:Eis *}),
an intermediate version was considered in \cite[Sect. 2.1]{BG}, which we will denote here by $\Eis_{!*}$. The definition
of $\Eis_{!*}$ uses the compactification of the morphism $\sfp$, introduced in \cite[Sect. 1.2]{BG}:
$$
\xy
(-20,0)*+{\Bun_G}="X";
(20,0)*+{\Bun_T.}="Y";
(0,20)*+{\ol\Bun_B}="Z";
(-20,20)*+{\Bun_B}="W";
{\ar@{->}_{\ol\sfp} "Z";"X"};
{\ar@{->}^{\ol\sfq} "Z";"Y"};
{\ar@{^{(}->}^r "W";"Z"};  
\endxy
$$

The assertion of \cite[Theorem 2.1.8]{BG} (for the longest element of the Weyl group) is:

\begin{thm}  \label{t:BG}
The following diagram of functors 
$$
\CD
\Dmod(\Bun_G) @>{\on{Id}_{\Dmod(\Bun_G) }}>>    \Dmod(\Bun_G)  \\
@A{\Eis_{!*}}AA    @AA{\Eis^-_{!*}}A   \\
\Dmod(\Bun_T) & &    \Dmod(\Bun_T) \\
@AAA    @AAA  \\
\Dmod(\Bun_T)^{\on{reg}} @>{\rho\on{-shift}}>>    \Dmod(\Bun_T)^{\on{reg}},
\endCD
$$
commutes up to a cohomological shift, where $\rho\text{-shift}$ is the functor
of translation by the point $2\rho(\Omega_X)$. \footnote{Here $2\rho:\BG_m\to T$ is the coweight
equal to the sum of positive coroots, and $\Omega_X\in \Pic(X)=\Bun_{\BG_m}$ is the canonical
line bundle on $X$.} 
\end{thm} 

\thmref{t:BG} is a geometric analog of the usual functional equation for Eisenstein series
in the theory of automorphic functions. 

\sssec{}

Let us emphasize the following points of difference between Theorems \ref{t:strange preview}
and \ref{t:BG}:

\begin{itemize}

\item \thmref{t:strange preview} compares the functors $\Eis^-_!$ and $(\on{CT}_*)^\vee$ that take
values in different categories, i.e., $\Dmod(\Bun_G)$ vs. $\Dmod(\Bun_G)^\vee$, whereas
in \thmref{t:BG} both $\Eis_{!*}$ and $\Eis^-_{!*}$ map to $\Dmod(\Bun_G)$. 

\item The vertical arrows in \thmref{t:strange preview} use geometrically different functors, 
while in \thmref{t:BG} these are functors of the same nature, i.e., $\Eis_{!*}$ and $\Eis^-_{!*}$. 

\item The upper horizontal arrow \thmref{t:strange preview} is the geometrically non-trivial
functor $\psId_{\Bun_G,!}$, while in \thmref{t:BG} it is the identity functor.

\item The lower horizontal arrow in \thmref{t:strange preview} for $M=T$ is isomorphic to
the identity functor, up to a cohomological shift, while in \thmref{t:BG} we have the functor
of $\rho$-shift. 

\item The commutation in \ref{t:strange preview} takes place on all of $\Dmod(\Bun_T)$,
whereas in \thmref{t:BG}, it only takes place on $\Dmod(\Bun_T)^{\on{reg}}$.

\end{itemize}

\ssec{Interaction with cuspidality}

There is yet one more set of results contained in this paper, which has to do
with the notion of cuspidality.

\sssec{}

The cuspidal subcategories
$$\Dmod(\Bun_G)_{\on{cusp}}\subset \Dmod(\Bun_G) \text{ and }
(\Dmod(\Bun_G)^\vee)_{\on{cusp}} \subset \Dmod(\Bun_G)^\vee$$
are defined as right-orthogonals of the subcategories generated by the essential
images of the functors
$$\Eis_!:\Dmod(\Bun_M)\to \Dmod(\Bun_G) \text{ and }
(\on{CT}_*)^\vee:\Dmod(\Bun_M)^\vee\to \Dmod(\Bun_G)^\vee,$$
respectively, for all \emph{proper} parabolics $P$ of $G$.

\sssec{}  \label{sss:naive intro}

Let us return to the setting of \secref{sss:dual first} are recall the ``naive" functor 
$$\psId_{\Bun_G,\on{naive}}:\Dmod(\Bun_G)^\vee\to \Dmod(\Bun_G),$$ 
see \secref{sss:dual first}. 

\medskip

As was mentioned in {\it loc.cit.}, the functor $\psId_{\Bun_G,\on{naive}}$ fails to be an equivalence
unless $G$ is a torus. However,  in \thmref{t:cusp naive equiv} we show:

\begin{thm} \label{t:cusp naive intro}
The restriction of the functor $\psId_{\Bun_G,\on{naive}}$ to 
$$(\Dmod(\Bun_G)^\vee)_{\on{cusp}} \subset \Dmod(\Bun_G)^\vee$$
defines an equivalence 
$$(\Dmod(\Bun_G)^\vee)_{\on{cusp}}\to \Dmod(\Bun_G)_{\on{cusp}}.$$
\end{thm} 

One can view \thmref{t:cusp naive intro} as expressing the fact that the objects of $(\Dmod(\Bun_G)^\vee)_{\on{cusp}}$
and $\Dmod(\Bun_G)_{\on{cusp}}$ are ``supported" on quasi-compact open substacks (see Propositions
\ref{p:!-supp cusp} and \ref{p:!-supp cusp co} for a precise statement). 

\sssec{}

In addition, in \corref{c:! and * on cusp} we show:

\begin{thm} \label{t:! and * on cusp intro}
The functors $$\psId_{\Bun_G,\on{naive}}|_{(\Dmod(\Bun_G)^\vee)_{\on{cusp}}} \text{ and }
\psId_{\Bun_G,!}|_{(\Dmod(\Bun_G)^\vee)_{\on{cusp}}}$$ are isomorphic up to a cohomological shift.
\end{thm} 

\thmref{t:! and * on cusp intro} is responsible for the fact that previous studies in geometric Langlands
correspondence that involved only cuspidal objects did not see the appearance of the functor 
$\psId_{\Bun_G,!}$ and one could afford to ignore the difference between $\Dmod(\Bun_G)$ and
$\Dmod(\Bun_G)^\vee$. In other words, usual manipulations with Verdier duality on
cuspidal objects did not produce wrong results.

\ssec{Structure of the paper}

\sssec{}

In \secref{s:functors} we recall the setting of \cite{DrGa3}, and list the various Eisenstein series and constant term 
functors for the usual category $\Dmod(\Bun_G)$. In fact there are two adjoint pairs:
$(\Eis_!,\on{CT}_*)$ and $(\on{CT}^\mu_!,\Eis^\mu_*)$, where in the latter pair the superscript
$\mu\in \pi_1(M)=\pi_0(\Bun_M)$ indicates that we are considering one connected component
of $\Bun_M$ at a time. 

\medskip

We recall the main result of \cite{DrGa3} that says that the functors $\on{CT}_*$ and $\on{CT}^-_!$
are canonically isomorphic. 

\medskip

Next, we consider the category $\Dmod(\Bun_G)_{\on{co}}$, which is nearly tautologically identified with
the category that we have so far denoted $\Dmod(\Bun_G)^\vee$, and introduce the corresponding 
Eisenstein series and constant term functors:
$$(\Eis_{\on{co},*},\on{CT}_{\on{co},?}) \text{ and } (\on{CT}^\mu_{\on{co},*},\Eis^\mu_{\on{co},?}),$$
where $\Eis_{\on{co}*}:=(\on{CT}_*)^\vee$, $\on{CT}^\mu_{\on{co},*}:=(\Eis^\mu_*)^\vee$.

\medskip

The functor $\on{CT}_{\on{co},?}$ is something that we do not know how to express 
in terms of the usual functors in the theory of D-modules; it can be regarded as a
\emph{non-standard} functor in the terminology of \cite[Sect. 3.3]{DrGa2}.

\medskip

A priori, the functor $\Eis^\mu_{\on{co},?}$ would also be a non-standard functor. However, the 
isomorphism $\on{CT}_*\simeq \on{CT}^-_!$ gives rise to an isomorphism 
$$\Eis_{\on{co},?}\simeq \Eis^-_{\on{co},*}.$$

\sssec{}

In \secref{s:naive} we recall the definition of the functor
$$\psId_{\Bun_G,\on{naive}}:\Dmod(\Bun_G)_{\on{co}}\to \Dmod(\Bun_G),$$
and show that it intertwines the functors $\Eis_{\on{co},*}$ and $\Eis_*$, and 
$\on{CT}_{\on{co},*}$ and $\on{CT}_*$, respectively. 

\medskip

The remainder of this section is devoted to the study of the subcategory
$$\Dmod(\Bun_G)_{\on{co,cusp}}\subset \Dmod(\Bun_G),$$
and the proof of \thmref{t:cusp naive intro}, which says that the functor
$\psId_{\Bun_G,\on{naive}}$ defines an equivalence from 
$\Dmod(\Bun_G)_{\on{co,cusp}}$ to $\Dmod(\Bun_G)_{\on{cusp}}\subset \Dmod(\Bun_G)$. 

\sssec{}

In \secref{s:pseudo} we introduce the functor 
$$\psId_{\Bun_G,!}:\Dmod(\Bun_G)_{\on{co}}\to \Dmod(\Bun_G),$$
and study its behavior vis-\`a-vis the functor $\psId_{\Bun_G,\on{naive}}$. 
The relation is expressed by \propref{p:diff}, whose proof is deferred to \cite{Sch}.
\propref{p:diff} essentially says that the difference between $\psId_{\Bun_G,!}$
and $\psId_{\Bun_G,\on{naive}}$ can be expressed in terms of the Eisenstein and
constant term functors for \emph{proper} parabolics. 

\medskip

We prove \thmref{t:! and * on cusp intro} that says that the functors $\psId_{\Bun_G,!}$
and $\psId_{\Bun_G,\on{naive}}$ are isomorphic (up to a cohomological shift), when
evaluated on cuspidal objects.

\sssec{}

In \secref{s:strange} we prove our ``strange" functional equation, i.e., \thmref{t:strange preview}.
The proof is basically a formal manipulation from the isomorphism $\on{CT}_*\simeq \on{CT}^-_!$.

\medskip

Having \thmref{t:strange preview}, we get control of the behavior of the functor $\psId_{\Bun_G,!}$
on the Eisenstein part of the category $\Dmod(\Bun_G)_{\on{co}}$. From here we deduce 
our main result, \thmref{t:miraculous preview}.

\ssec{Conventions}

The conventions in this paper follow those adopted in \cite{DrGa2}. We refer the reader to {\it loc.cit.}
for a review of the theory of DG categories (freely used in this paper), and the theory of D-modules
on stacks. 

\ssec{Acknowledgements}

The author would like to express his gratitude to V. ~Drinfeld. The main results of this paper, in particular, \thmref{t:strange preview}, 
were obtained in collaboration with him. Even more crucially, the very idea
of the pseudo-identity functor $\psId_{\Bun_G,!}$ is the invention of his. 

\medskip

The author is supported by NSF grant DMS-1063470.

\section{The inventory of categories and functors}  \label{s:functors}

\ssec{Eisenstein series and constant term functors}

\sssec{}

Let $P$ be a parabolic in $G$ with Levi quotient $M$. For 
$$\mu\in \pi_1(M)\simeq \pi_0(\Bun_M)\simeq \pi_0(\Bun_P),$$
let $\Bun_M^\mu$ (resp., $\Bun_P^\mu$) denote the corresponding
connected component of $\Bun_M$ (resp., $\Bun_P$).

\sssec{}

Consider the diagram
\begin{equation} \label{e:basic diagram}
\xy
(-15,0)*+{\Bun_G}="X";
(15,0)*+{\Bun^\mu_M.}="Y";
(0,15)*+{\Bun^\mu_P}="Z";
{\ar@{->}_{\sfp} "Z";"X"};
{\ar@{->}^{\sfq} "Z";"Y"};
\endxy
\end{equation}

We consider the functor
$$\on{CT}^\mu_*:\Dmod(\Bun_G)\to \Dmod(\Bun^\mu_M), \quad \on{CT}^\mu_*=\sfq_*\circ \sfp^!.$$

\sssec{}

According to \cite[Corollary 1.1.3]{DrGa3}, the functor $\on{CT}^\mu_*$ admits a left adjoint, denoted
by $\Eis^\mu_!$. Explicitly, 
$$\Eis^\mu_!=\sfp_!\circ \sfq^*.$$

The above expression has to be understood as follows: the functor $$\sfq^*:\Dmod(\Bun^\mu_M)\to \Dmod(\Bun_P^\mu)$$
is defined (because the morphism $\sfq$ is smooth), and the partially defined functor $\sfp_!$, left adjoint to $\sfp^!$,
is defined on the essential image of $\sfq^*$ by \cite[Proposition 1.1.2]{DrGa3}. 

\sssec{}

We define the functor $\on{CT}_*:\Dmod(\Bun_G)\to \Dmod(\Bun_M)$ as 
$$\on{CT}_*\simeq \underset{\mu}\bigoplus\, \on{CT}_*^\mu.$$

\medskip

We define the functor $\Eis_!:\Dmod(\Bun_M)\to \Dmod(\Bun_G)$ as 
$$\Eis_!\simeq \underset{\mu}\bigoplus\, \Eis_!^\mu.$$

\begin{lem}
The functor $\Eis_!$ is the left adjoint of $\on{CT}_*$.
\end{lem}

\begin{proof}
Follows from the fact that
$$\underset{\mu}\bigoplus\, \on{CT}_*^\mu\simeq \underset{\mu}\prod\, \on{CT}_*^\mu.$$
\end{proof}

\sssec{}  \label{sss:Eis *}

We now consider the functor $\Eis^\mu_*:\Dmod(\Bun^\mu_M)\to \Dmod(\Bun_G)$, defined
as
$$\Eis^\mu_*=\sfp_*\circ \sfq^!.$$

\medskip

We let $\Eis^{\mu,-}_*$ and $\on{CT}_*^{\mu,-}$ be similarly defined functors when instead of $P$
we use the opposite parabolic $P^-$ (we identify the Levi quotients of $P$ and $P^-$ via the 
isomorphsim $M\simeq P\cap P^-$). 

\medskip

The following is the main result of \cite{DrGa3}:

\begin{thm} \label{t:weird adj}
The functor $\Eis^\mu_*$ canonically identifies with the \emph{right} adjoint 
of $\on{CT}_*^{\mu,-}$.
\end{thm}

\sssec{}

We will use the notation 
$$\on{CT}^\mu_!:\Dmod(\Bun_G)\to \Dmod(\Bun^\mu_M)$$
for the left adjoint of $\Eis^\mu_*$. If $\CF\in \Dmod(\Bun_G)$ is such that the partially defined
left adjoint $\sfp^*$ of $\sfp_*$ is defined on $\CF$, then we have
$$\on{CT}^\mu_!(\CF)\simeq \sfq_!\circ \sfp^*(\CF).$$
(The functor $\sfq_!$, left adjoint to $\sfq^!$, is well-defined by \cite[Sect. 3.1.5]{DrGa3}.)

\medskip

Hence, \thmref{t:weird adj} can be reformulated as saying that $\on{CT}^\mu_!$ exists and is canonically
isomorphic to $\on{CT}_*^{\mu,-}$.

\sssec{}

We define the functor $\on{CT}_!:\Dmod(\Bun_G)\to \Dmod(\Bun_M)$ as 
$$\on{CT}_!\simeq \underset{\mu}\bigoplus\, \on{CT}_!^\mu,$$
so $\on{CT}_!\simeq \on{CT}^-_*$. 

\medskip

We define the functor $\Eis_*:\Dmod(\Bun_M)\to \Dmod(\Bun_G)$ as 
$$\Eis_*\simeq \underset{\mu}\bigoplus\, \Eis_*^\mu.$$

\medskip

We note, however, that it is \emph{no longer true} that $\Eis_*$ is the right adjoint of 
$\on{CT}_!$. (Rather, the right adjoint of $\on{CT}_!$ is the functor 
$\underset{\mu}\prod\, \Eis_*^\mu$.)

\medskip

In fact, one can show that the functor $\Eis_*$ \emph{does not admit} a left adjoint,
see \cite[Sect. 1.2.1]{DrGa3}. 

\ssec{The dual category}  \label{ss:dual}

\sssec{}

Let $\on{op-qc}(G)$ denote the poset of open substacks $U\overset{j}\hookrightarrow \Bun_G$ such that the intersection of $U$
with every connected component of $\Bun_G$ is quasi-compact. 

\medskip

We have 
\begin{equation} \label{e:DBunG as limit}
\Dmod(\Bun_G)\simeq \underset{U\in \on{op-qc}(G)}{\underset{\longleftarrow}{{\on{lim}}}}\, \Dmod(U),
\end{equation} 
where for $U_1\overset{j_{1,2}}\hookrightarrow U_2$, the corresponding functor
$\Dmod(U_2)\to \Dmod(U_1)$ is $j_{1,2}^*$ (see, e.g., \cite[Lemma 2.3.2]{DrGa2} for the proof). 

\medskip

Under the equivalence \eqref{e:DBunG as limit}, for 
$$(U\overset{j}\hookrightarrow \Bun_G) \in \on{op-qc}(G),$$
the tautological evaluation functor $\Dmod(\Bun_G)\to \Dmod(U)$ is $j^*$. 

\sssec{}

The following DG category was introduced in \cite[Sect. 4.3.3]{DrGa2}:
\begin{equation} \label{e:DBunG dual as colimit}
\Dmod(\Bun_G)_{\on{co}}:=\underset{U\in \on{op-qc}(G)}{\underset{\longrightarrow}{{\on{colim}}}}\, \Dmod(U),
\end{equation} 
where for $U_1\overset{j_{1,2}}\hookrightarrow U_2$, the corresponding functor
$\Dmod(U_1)\to \Dmod(U_2)$ is $(j_{1,2})_*$, and where the colimit is taken in the category
of cocomplete DG categories and continuous functors. 

\medskip

For $(U\overset{j}\hookrightarrow \Bun_G) \in \on{op-qc}(G)$ we let
$j_{\on{co},*}$ denote the tautological functor
\begin{equation} \label{e:j co}
j_{\on{co},*}:\Dmod(U)\to \Dmod(\Bun_G)_{\on{co}}.
\end{equation} 

\sssec{}

Verdier duality functors
$$\bD_{U}:\Dmod(U)^\vee\simeq \Dmod(U)$$
for $U\in  \on{op-qc}(G)$
and the identifications
$$((j_{1,2})_*)^\vee\simeq (j_{1,2})^*,\quad U_1\overset{j_{1,2}}\hookrightarrow U_2$$
give rise to an identification

\begin{equation} \label{e:dual of co prel}
\on{Funct}_{\on{cont}}(\Dmod(\Bun_G)_{\on{co}},\Vect)\simeq \Dmod(\Bun_G).
\end{equation}

Now, the main result of \cite{DrGa2}, namely, Theorem 4.1.8, implies: 

\begin{thm}
The category $\Dmod(\Bun_G)_{\on{co}}$ is compactly generated (and, in particular, dualizable).
\end{thm}

\begin{proof}

The truncatability of $\Bun_G$ means that in the presentation of 
$\Dmod(\Bun_G)_{\on{co}}$ as a colimit \eqref{e:DBunG dual as colimit}, we can replace the index poset 
$\on{op-qc}(G)$ by a \emph{cofinal} poset that consists of quasi-compact open substacks that are
\emph{co-truncative}.

\medskip

Then the resulting colimit
$$\underset{U}{\underset{\longrightarrow}{{\on{colim}}}}\, \Dmod(U)$$
consists of compactly generated categories and functors \emph{that preserve compactness}.
In this case, the resulting colimit category is compactly generated, e.g., by 
\cite[Corollary 1.9,4]{DrGa2}.

\end{proof} 

From \eqref{e:dual of co prel}, and knowing that $\Dmod(\Bun_G)_{\on{co}}$ is dualizable, 
we obtain a canonical identification
\begin{equation} \label{e:dual of co}
\bD_{\Bun_G}:\Dmod(\Bun_G)^\vee\simeq \Dmod(\Bun_G)_{\on{co}}.
\end{equation}

Under this identification, for $(U\overset{j}\hookrightarrow \Bun_G) \in \on{op-qc}(G)$ we have
the following canonical identification of functors
$$(j_{\on{co},*})^\vee\simeq j^*.$$

\sssec{}

Similar constructions and notation apply when instead of all of $\Bun_G$ we consider one of its connected
components $\Bun_G^\lambda$, $\lambda\in \pi_1(G)$.

\ssec{Dual, adjoint and conjugate functors} 

\sssec{}

Let $\bC_1$ and $\bC_2$ be two DG categories, and let
$$\sF:\bC_1\rightleftarrows \bC_2:\sG$$
be a pair of \emph{continuous} mutually adjoint functors.

\sssec{}  \label{sss:dual and adj}

By passing to dual functors, the adjunction data
$$\on{Id}_{\bC_1}\to \sG\circ \sF \text{ and } \sF\circ \sG\to \on{Id}_{\bC_2}$$
gives rise to
$$\on{Id}_{\bC^\vee_1}\to \sF^\vee\circ \sG^\vee \text{ and } \sG^\vee\circ \sF^\vee\to \on{Id}_{\bC^\vee_2},$$
making 
$$\sG^\vee:\bC^\vee_1\rightleftarrows \bC^\vee_2:\sF^\vee$$
into a pair of adjoint functors.

\sssec{}

Assume now that $\bC_1$ is compactly generated. In this case, the fact that the right adjoint $\sG$ of $\sF$
is continuous is equivalent to the fact that $\sF$ preserves compactness. I.e., it defines a functor
between non-cocomplete DG categories 
$$\bC_1^c\to \bC_2^c,$$
and hence, by passing to the opposite categories, a functor
\begin{equation} \label{e:functor on compact op}
(\bC_1^c)^{\on{op}}\to (\bC_2^c)^{\on{op}},
\end{equation}

\medskip

Following \cite[Sect. 1.5]{Ga2}, we let
$$\sF^{\on{op}}:\bC_1^\vee\to \bC_2^\vee$$
denote the functor obtained as the composition of:

\smallskip

\noindent{(i)} The identification $\bC_1^\vee \simeq \on{Ind}((\bC_1^c)^{\on{op}})$;

\smallskip

\noindent{(ii)} The ind-extension $\on{Ind}((\bC_1^c)^{\on{op}})\to \on{Ind}((\bC_2^c)^{\on{op}})$ of \eqref{e:functor on compact op};

\smallskip

\noindent{(iii)} The fully faithful embedding $(\bC_2^c)^{\on{op}}\hookrightarrow \bC_2^\vee$. 

\medskip

We call $\sF^{\on{op}}$ the functor \emph{conjugate} to $\sF$.

\sssec{}

The following is \cite[Lemma 1.5.3]{Ga2}:

\begin{lem}  \label{l:conj}
We have a canonical isomorphism of functors $\sF^{\on{op}}\simeq \sG^\vee$.
\end{lem}

\ssec{Dual Eisenstein series and constant term functors}

\sssec{}  \label{sss:dual Eis}

We define the functor
$$\Eis^\mu_{\on{co},*}:\Dmod(\Bun^\mu_M)_{\on{co}}\to \Dmod(\Bun_G)_{\on{co}}$$
as
$$\Eis^\mu_{\on{co},*}\simeq (\on{CT}^\mu_*)^\vee$$
under the identifications \eqref{e:dual of co} and
$$\bD_{\Bun_M^\mu}:\Dmod(\Bun_M^\mu)^\vee\simeq \Dmod(\Bun_M^\mu)_{\on{co}}.$$

\medskip

We define
$$\Eis_{\on{co},*}:\Dmod(\Bun_M)_{\on{co}}\to \Dmod(\Bun_G)_{\on{co}}$$
as 
$$\Eis_{\on{co},*}:=\underset{\mu}\bigoplus\, \Eis^\mu_{\on{co},*}\simeq (\on{CT}_*)^\vee.$$

\medskip

Note that by \lemref{l:conj}, we have:

\begin{cor}  \label{c:Eis op}
There are canonical isomorphisms
$$\Eis_{\on{co},*}\simeq (\Eis_!)^{\on{op}} \text{ and } \Eis^\mu_{\on{co},*}\simeq (\Eis^\mu_!)^{\on{op}}.$$
\end{cor}

\sssec{}

We define the functor
$$\on{CT}^\mu_{\on{co},*}:\Dmod(\Bun_G)_{\on{co}}\to \Dmod(\Bun^\mu_M)_{\on{co}}$$
as
$$\on{CT}^\mu_{\on{co},*}\simeq (\Eis^\mu_*)^\vee.$$

\medskip

We define 
$$\on{CT}_{\on{co},*}:\Dmod(\Bun_G)_{\on{co}}\to \Dmod(\Bun_M)_{\on{co}}$$
as 
$$\on{CT}_{\on{co},*}:=\underset{\mu}\bigoplus\, \on{CT}^\mu_{\on{co},*}\simeq (\Eis_*)^\vee.$$

\medskip

From \lemref{l:conj}, we obtain:

\begin{cor}
There is a canonical isomorphism
$$\on{CT}^\mu_{\on{co},*}\simeq (\on{CT}_!^\mu)^{\on{op}}.$$
\end{cor}

\sssec{}

Define also
$$\on{CT}^\mu_{\on{co},?}:=(\Eis^\mu_!)^\vee \text{ and } \on{CT}_{\on{co},?}:=(\Eis_!)^\vee\simeq \underset{\mu}\bigoplus\, \on{CT}^\mu_{\on{co},?},$$
$$\Eis^\mu_{\on{co},?}:=(\on{CT}^\mu_!)^\vee \text{ and } \Eis_{\on{co},?}:=(\on{CT}_!)^\vee\simeq \underset{\mu}\bigoplus\, \Eis^\mu_{\on{co},?}.$$

By \secref{sss:dual and adj}, we obtain the following pairs of adjoint functors
$$\Eis^\mu_{\on{co},*}:\Dmod(\Bun^\mu_M)_{\on{co}}\rightleftarrows \Dmod(\Bun_G)_{\on{co}}:\on{CT}^\mu_{\on{co},?},$$
$$\Eis_{\on{co},*}:\Dmod(\Bun_M)_{\on{co}}\rightleftarrows \Dmod(\Bun_G)_{\on{co}}:\on{CT}_{\on{co},?},$$
and
$$\on{CT}^\mu_{\on{co},*}:\Dmod(\Bun_G)_{\on{co}} \rightleftarrows  \Dmod(\Bun^\mu_M)_{\on{co}}:\Eis^\mu_{\on{co},?}.$$

\sssec{}

Finally, from \thmref{t:weird adj}, we obtain:

\begin{cor}
There are canonical isomorphisms of functors
$$\Eis^\mu_{\on{co},?}\simeq \Eis^{\mu,-}_{\on{co},*}, \quad \Eis_{\on{co},?}\simeq \Eis^-_{\on{co},*}$$
and
$$\on{CT}^\mu_{\on{co},*} \simeq (\on{CT}^{\mu,-}_*)^{\on{op}}.$$
\end{cor} 

To summarize, we also obtain an adjunction
$$\on{CT}^\mu_{\on{co},*}:\Dmod(\Bun_G)_{\on{co}} \rightleftarrows  \Dmod(\Bun^\mu_M)_{\on{co}}:\Eis^{\mu,-}_{\on{co},*}.$$

\sssec{}

We can ask the following question: does the functor $\on{CT}^\mu_{\on{co},*}$ admit a \emph{left} adjoint? The answer
is ``no":

\begin{proof}
If $\on{CT}^\mu_{\on{co},*}$ had admitted a left adjoint, by \secref{sss:dual and adj}, the functor $\Eis^\mu_*$ would have admitted
a \emph{continuous} right adjoint. However, this is not the case, since the functor 
$$\Eis^\mu_*:\Dmod(\Bun^\mu_M)\to \Dmod(\Bun_G)$$
\emph{does not} preserve compactness. 
\end{proof}

\ssec{Explicit description of the dual functors}    \label{ss:dual of CT}

\sssec{}

For $(U\overset{j}\hookrightarrow \Bun_G) \in \on{op-qc}(G)$ we consider the functor
$j^*:\Dmod(\Bun_G)\to \Dmod(U)$, and its right adjoint $j_*$.

\medskip

Define 
$$j^*_{\on{co}}:\Dmod(\Bun_G)_{\on{co}}\to \Dmod(U)$$
as 
$$j^*_{\on{co}}:=(j_*)^\vee.$$

By \secref{sss:dual and adj}, the functors
$$j^*_{\on{co}}:\Dmod(\Bun_G)_{\on{co}}\rightleftarrows  \Dmod(U):j_{\on{co},*}$$
form an adjoint pair, where $j_{\on{co},*}$ is as in \eqref{e:j co}.

\begin{lem} \label{l:j ff}
The functor $j_{\on{co},*}$ is fully faithful.
\end{lem}

\begin{proof}
We need to show that the co-unit of the adjunction
$$j^*_{\on{co}}\circ j_{\on{co},*}\to \on{Id}_{\Dmod(U)}$$
is an isomorphism. But this follows from the fact that the corresponding map between
the dual functors, i.e.,
$$j^*\circ j_*\to  \on{Id}_{\Dmod(U)},$$
is an isomorphism (the latter because $j_*:\Dmod(U)\to \Dmod(\Bun_G)$ is fully faithful).
\end{proof} 

\sssec{}

By the definition of $\Dmod(\Bun_G)_{\on{co}}$, the functor $j^*_{\on{co}}$ amounts
to a compatible family of functors
$$j^*_{\on{co}}\circ (j_1)_{\on{co},*}:\Dmod(U_1)\to \Dmod(U)$$
for $(U_1\overset{j_1}\hookrightarrow \Bun_G) \in \on{op-qc}(G)$.

\medskip

It is easy to see from the definitions that
$$j^*_{\on{co}}\circ (j_1)_{\on{co},*}\simeq (j'_1)_*\circ (j')^*,$$
where
$$
\CD
U\cap U_1 @>{j'_1}>>  U  \\
@V{j'}VV  @VV{j}V  \\
U_1  @>{j_1}>>  \Bun_G.
\endCD
$$

\sssec{}

Again, by the definition of the category $\Dmod(\Bun_M)_{\on{co}}$, the functor
$\Eis_{\on{co},*}$ amounts to a compatible family of functors
$$\Eis_{\on{co},*}\circ (j_M)_{\on{co},*}:\Dmod(U_M)\to \Dmod(\Bun_G)_{\on{co}}$$
for $(U_M\overset{j_M}\hookrightarrow \Bun_M) \in \on{op-qc}(M)$.

\medskip

We now claim:

\begin{prop}  \label{p:Eis co expl}
For a given $(U_M\overset{j_M}\hookrightarrow \Bun_M) \in \on{op-qc}(M)$, let 
$(U_G\overset{j_G}\hookrightarrow \Bun_G) \in \on{op-qc}(G)$ be such that
$$\sfp(\sfq^{-1}(U_M))\subset U_G.$$
Then there is a canonical isomorphism
$$\Eis_{\on{co},*}\circ (j_M)_{\on{co},*}\simeq (j_G)_{\on{co},*}\circ (j_G)^*\circ \Eis_*\circ (j_M)_*:\quad
\Dmod(U_M)\to \Dmod(\Bun_G)_{\on{co}}.$$
\end{prop}

\begin{proof}

First, we claim that there is a canonical isomorphism
\begin{equation} \label{e:Eis j}
\Eis_{\on{co},*}\circ (j_M)_{\on{co},*}\simeq
(j_G)_{\on{co},*}\circ (j_G)_{\on{co}}^*\circ \Eis_{\on{co},*}\circ (j_M)_{\on{co},*}.
\end{equation}

Indeed, \eqref{e:Eis j} follows by passing to dual functors in the isomorphism
$$(j_M)^*\circ \on{CT}_*\simeq (j_M)^*\circ \on{CT}_*\circ (j_G)_*\circ (j_G)^*,$$
where the latter follows by base change from the definition.

\medskip

Hence, it remains to establish a canonical isomorphism of functors
$$(j_G)_{\on{co}}^*\circ \Eis_{\on{co},*}\circ (j_M)_{\on{co},*}\simeq (j_G)^*\circ \Eis_*\circ (j_M)_*,\quad
\Dmod(U_M)\to \Dmod(U_G),$$
i.e., an isomorphism
$$((j_M)^*\circ \on{CT}_*\circ (j_G)_*)^\vee \simeq (j_G)^*\circ \Eis_*\circ (j_M)_*.$$

However, the latter amounts to pull-push along the diagram
$$
\xy
(-15,0)*+{U_G}="X";
(15,0)*+{U_M}="Y";
(0,15)*+{U_P}="Z";
(-35,0)*+{\Bun_G}="W";
(35,0)*+{\Bun_M,}="U";
{\ar@{->}_{\sfp|_{U_P}} "Z";"X"};
{\ar@{->}^{\sfq|_{U_P}} "Z";"Y"};
{\ar@{_{(}->}_{j_P} "X";"W"};
{\ar@{^{(}->}^{j_M} "Y";"U"};
\endxy
$$
where $U_P:=\sfq^{-1}(U_M)$. 

\end{proof}

\sssec{}

The functor $\on{CT}_{\on{co},*}$ amounts to a compatible family of functors
$$\on{CT}_{\on{co},*}\circ (j_G)_{\on{co},*}:\Dmod(U_G)\to \Dmod(\Bun_M)_{\on{co}}$$
for $(U_G\overset{j_G}\hookrightarrow \Bun_G) \in \on{op-qc}(G)$.

\medskip

In a similar way to \propref{p:Eis co expl}, we have:

\begin{prop}  \label{p:CT co expl}
For a given $(U_G\overset{j_G}\hookrightarrow \Bun_G) \in \on{op-qc}(G)$, let 
$(U_M\overset{j_M}\hookrightarrow \Bun_M) \in \on{op-qc}(M)$ be such that
$$\sfq(\sfp^{-1}(U_G))\subset U_M.$$
Then there is a canonical isomorphism
$$\on{CT}_{\on{co},*}\circ (j_G)_{\on{co},*}\simeq (j_M)_{\on{co},*}\circ (j_M)^*\circ \on{CT}_*\circ (j_G)_*:\quad
\Dmod(U_G)\to \Dmod(\Bun_M)_{\on{co}}.$$
\end{prop}

\section{Interaction with the naive pseudo-identity and cuspidality}  \label{s:naive}

\ssec{The naive pseudo-identity functor}

\sssec{}

The following functor
$$\psId_{\Bun_G,\on{naive}}:\Dmod(\Bun_G)_{\on{co}}\to \Dmod(\Bun_G)$$
was introduced in \cite[Sect. 4.4.2]{DrGa2}:

\medskip

For $(U_G\overset{j_G}\hookrightarrow \Bun_G) \in \on{op-qc}(G)$, the composition
$$\psId_{\Bun_G,\on{naive}}\circ j_{\on{co},*}:\Dmod(U_G)\to \Dmod(\Bun_G)$$
is by definition the functor $j_*$.

\begin{rem}
The functor $\psId_{\Bun_G,\on{naive}}$ is very far from being an equivalence, unless $G$ is a torus. 
For example, in \cite[Theorem 7.7.2]{Ga2}, a particular object of $\Dmod(\Bun_G)_{\on{co}}$ was constructed,
which belongs to $\on{ker}(\psId_{\Bun_G,\on{naive}})$, as soon as the semi-simple part of $G$ is non-trivial. 
\end{rem}

\sssec{}  \label{sss:naive by kernel}

Recall the equivalence:
\begin{multline*}  %\label{e:functors}
\on{Funct}_{\on{cont}}(\Dmod(\Bun_G)_{\on{co}},\Dmod(\Bun_G))\simeq
(\Dmod(\Bun_G)_{\on{co}})^\vee\otimes \Dmod(\Bun_G)\simeq  \\
\simeq \Dmod(\Bun_G)\otimes \Dmod(\Bun_G)\simeq \Dmod(\Bun_G\times \Bun_G).
\end{multline*}

\medskip

According to \cite[Sect. 4.4.3]{DrGa2}, the functor $\psId_{\Bun_G,\on{naive}}$ corresponds to the object
$$(\Delta_{\Bun_G})_*(\omega_{\Bun_G})\in \Dmod(\Bun_G\times \Bun_G),$$
where $\Delta_{\Bun_G}$ denotes the diagonal morphism on $\Bun_G$, and $\omega_\CY$ is the dualizing object
on a stack $\CY$ (we take $\CY=\Bun_G$).

\medskip

From here we obtain:

\begin{lem} \label{l:naive self-dual} There exists a canonical isomorphism
$\psId_{\Bun_G,\on{naive}}^\vee\simeq \psId_{\Bun_G,\on{naive}}$.
\end{lem}

\begin{proof}
This expresses the fact that $(\Delta_{\Bun_G})_*(\omega_{\Bun_G})$ is equivariant with respect to
the flip automorphism of $\Dmod(\Bun_G\times \Bun_G)$.
\end{proof} 

\begin{cor} \label{c:naive and res}
For  $(U_G\overset{j_G}\hookrightarrow \Bun_G) \in \on{op-qc}(G)$, we have a canonical isomorphism:
$$j^*\circ \psId_{\Bun_G,\on{naive}}\simeq j^*_{\on{co}}.$$
\end{cor}

\begin{proof}
Obtained by passing to the dual functors is
$$\psId_{\Bun_G,\on{naive}}\circ j_{\on{co},*}\simeq j_*.$$
\end{proof}

\sssec{}

We now claim:

\begin{prop} \label{p:naive and functors}
There are canonical isomorphisms
$$\psId_{\Bun_G,\on{naive}}\circ \Eis_{\on{co},*}\simeq \Eis_*\circ \psId_{\Bun_M,\on{naive}}$$
and
$$\psId_{\Bun_M,\on{naive}}\circ \on{CT}_{\on{co},*}\simeq \on{CT}_*\circ \psId_{\Bun_G,\on{naive}}.$$
\end{prop}

\begin{proof}

We will prove the first isomorphism, while the second one is similar. 

\medskip

By definition, we need to construct a compatible family of isomorphisms of functors
$$\psId_{\Bun_G,\on{naive}}\circ \Eis_{\on{co},*}\circ (j_M)_{\on{co},*}
\simeq  \Eis_*\circ \psId_{\Bun_M,\on{naive}} \circ (j_M)_{\on{co},*}$$
for $(U_M\overset{j_M}\hookrightarrow \Bun_M) \in \on{op-qc}(M)$.

\medskip

For a given $U_M$, let $U_G$ be as in \propref{p:Eis co expl}. We rewrite
\begin{multline*}
\psId_{\Bun_G,\on{naive}}\circ \Eis_{\on{co},*}\circ (j_M)_{\on{co},*}
\simeq  
\psId_{\Bun_G,\on{naive}}\circ  (j_G)_{\on{co},*}\circ (j_G)^*\circ \Eis_*\circ (j_M)_*\simeq \\
\simeq  (j_G)_*\circ (j_G)^*\circ \Eis_*\circ (j_M)_*.
\end{multline*} 

However, it is easy to see that for the above choice of $U_G$, the natural map 
$$\Eis_*\circ (j_M)_*\to (j_G)_*\circ (j_G)^*\circ \Eis_*\circ (j_M)_*$$
is an isomorphism.

\medskip

Now, by definition,
$$\Eis_*\circ \psId_{\Bun_M,\on{naive}} \circ (j_M)_{\on{co},*}\simeq \Eis_*\circ (j_M)_*,$$
and the assertion follows. (It is clear that these isomorphisms are independent of the choice
of $U_G$, and hence are compatible under $(U_1)_M\hookrightarrow (U_2)_M$.)

\end{proof}

\ssec{Cuspidality}

\sssec{}

Recall that in \cite[Sect. 1.4]{DrGa3} the full subcategory
$$\Dmod(\Bun_G)_{\on{cusp}}\subset \Dmod(\Bun_G)$$
was defined as the intersection of the kernels of the functors $\on{CT}_*$ for 
\emph{all proper parabolic subgroups} $P\subset G$.

\medskip

Equivalently, let
$$\Dmod(\Bun_G)_{\Eis}\subset \Dmod(\Bun_G)$$
be the full subcategory, generated by the essential images of the functors
$\Eis_!$ for all proper parabolics. From the $(\Eis_!,\on{CT}_*)$-adjunction,
we obtain 
$$\Dmod(\Bun_G)_{\on{cusp}}=\left(\Dmod(\Bun_G)_{\Eis}\right)^\perp.$$

\sssec{}

We let
$$\Dmod(\Bun_G)_{\on{co,Eis}}\subset \Dmod(\Bun_G)_{\on{co}}$$
be the full subcategory generated by the essential images of the functors
$$\Eis_{\on{co},*}:\Dmod(\Bun_M)_{\on{co}}\to \Dmod(\Bun_G)_{\on{co}}.$$

\medskip

We define 
$$\Dmod(\Bun_G)_{\on{co,cusp}}:=\left(\Dmod(\Bun_G)_{\on{co,Eis}}\right)^\perp.$$

Equivalently, $\Dmod(\Bun_G)_{\on{co,cusp}}$ is the intersection of the kernels
of the functors $\on{CT}_{\on{co},?}$ for all proper parabolics. 

\sssec{}

From \corref{c:Eis op} we obtain:

\begin{cor}  \label{c:cusp and dual} \hfill

\smallskip

\noindent{\em(1)} An object of $\Dmod(\Bun_G)_{\on{co}}$ is cuspidal if and only if its pairing
with every object of $\Dmod(\Bun_G)_{\Eis}$ is zero under the canonical map
$$\langle-,-\rangle_{\Bun_G}:\Dmod(\Bun_G)\times \Dmod(\Bun_G)_{\on{co}}\to \Vect$$
corresponding to $\bD_{\Bun_G}$. 

\smallskip

\noindent{\em(2)} The identification $\bD_{\Bun_G}:\Dmod(\Bun_G)^\vee\simeq \Dmod(\Bun_G)_{\on{co}}$ induces identifications
$$(\Dmod(\Bun_G)_{\Eis})^\vee\simeq \Dmod(\Bun_G)_{\on{co,Eis}} \text{ and }
(\Dmod(\Bun_G)_{\on{cusp}})^\vee\simeq \Dmod(\Bun_G)_{\on{co,cusp}}.$$ 
\end{cor}

\begin{rem}
We will see shortly that $\Dmod(\Bun_G)_{\on{co,cusp}}$ belongs to the intersection
of the kernels of the functors $\on{CT}_{\on{co},*}$ for all proper parabolics. But this
inclusion is strict. For example fr $G=SL_2$, the object from \cite[Theorem 7.7.2]{Ga2} 
belongs to $\on{CT}_{\on{co},*}$  (there is only one parabolic to consider), but it does
not belong to $\Dmod(\Bun_G)_{\on{co,cusp}}$ .
\end{rem}

\sssec{}

Our goal for the rest of this section is to prove:

\begin{thm}  \label{t:cusp naive equiv}
The restriction of the functor $\psId_{\Bun_G,\on{naive}}$ to
$$\Dmod(\Bun_G)_{\on{co,cusp}}\subset \Dmod(\Bun_G)_{\on{co}}$$ takes values in
$\Dmod(\Bun_G)_{\on{cusp}}\subset \Dmod(\Bun_G)$, and defines an equivalence
$$\Dmod(\Bun_G)_{\on{co,cusp}}\to \Dmod(\Bun_G)_{\on{cusp}}.$$
\end{thm}

\ssec{Support of cuspidal objects}

\sssec{}

The following crucial property of $\Dmod(\Bun_G)_{\on{cusp}}$ was established in
\cite[Proposition 1.4.6]{DrGa3}:

\begin{prop}  \label{p:!-supp cusp}
There exists an element $(\CU_G\overset{\jmath_G}\hookrightarrow \Bun_G) \in \on{op-qc}(G)$,
such that for any $\CF\in \Dmod(\Bun_G)_{\on{cusp}}$, the maps
$$(\jmath_G)_!\circ (\jmath_G)^*(\CF)\to \CF\to (\jmath_G)_*\circ (\jmath_G)^*(\CF)$$
are isomorphisms.
\end{prop}

\sssec{}

We now claim that a parallel phenomenon takes place for $\Dmod(\Bun_G)_{\on{co,cusp}}$: 

\begin{prop} \label{p:!-supp cusp co}
For any $\CF\in \Dmod(\Bun_G)_{\on{co,cusp}}$, the map
$$\CF\to (\jmath_G)_{\on{co},*}\circ (\jmath_G)^*_{\on{co}}(\CF)$$
is an isomorphism.
\end{prop} 

\begin{proof}

We need to show that the map from the tautological embedding
\begin{equation}  \label{e:emb}
\Dmod(\Bun_G)_{\on{co,cusp}}\overset{\be_{\on{co}}}\hookrightarrow \Dmod(\Bun_G)_{\on{co}}
\end{equation} 
to the composition 
$$\Dmod(\Bun_G)_{\on{co,cusp}}\overset{\be_{\on{co}}}\hookrightarrow \Dmod(\Bun_G)_{\on{co}}\overset{(\jmath_G)_{\on{co}}^*}\longrightarrow
\Dmod(\CU_G) \overset{(\jmath_G)_{\on{co},*}}\longrightarrow \Dmod(\Bun_G)_{\on{co}}$$
is an isomorphism.

\medskip

Note that in terms of the identification of \corref{c:cusp and dual}(b), the dual of the embedding $\be_{\on{co}}$ of \eqref{e:emb}
is the functor
\begin{equation}  \label{e:emb right}
\bbf:\Dmod(\Bun_G)\to \Dmod(\Bun_G)_{\on{cusp}},
\end{equation}
left adjoint to the tautological embedding $\Dmod(\Bun_G)_{\on{cusp}}\overset{\be}\hookrightarrow \Dmod(\Bun_G)$. 

\medskip

Hence, by duality, we need to show that the functor \eqref{e:emb right} maps isomorphically to the composition
$$\Dmod(\Bun_G)\overset{(\jmath_G)^*}\longrightarrow
\Dmod(\CU_G) \overset{(\jmath_G)_*}\longrightarrow \Dmod(\Bun_G)\overset{\bbf}\longrightarrow \Dmod(\Bun_G)_{\on{cusp}}.$$

The latter is equivalent to the fact that any $\CF'\in \Dmod(\Bun_G)$ for which $\jmath_G^*(\CF')=0$, is left-orthogonal to
$\Dmod(\Bun_G)_{\on{cusp}}$. However, this follows from the isomorphism
$$\CF\to (\jmath_G)_*\circ (\jmath_G)^*(\CF),\quad \CF\in \Dmod(\Bun_G)_{\on{cusp}}$$
of \propref{p:!-supp cusp}.

\end{proof}

\ssec{Description of the cuspidal category}

\sssec{}

We claim: 

\begin{prop} \label{p:cusp co via CT}
Let $\CF\in \Dmod(\Bun_G)_{\on{co}}$ be such that there exists
$(U\overset{j}\hookrightarrow \Bun_G) \in \on{op-qc}(G)$ such that the map
$$\CF\to j_{\on{co},*}\circ j^*_{\on{co}}(\CF)$$
is an isomorphism. Then $\CF\in \Dmod(\Bun_G)_{\on{co,cusp}}$ if and only if 
$\on{CT}_{\on{co},*}(\CF)=0$ for all proper parabolics.
\end{prop} 

\begin{proof}

Recall that 
$$\langle-,-\rangle_{\Bun_G}: \Dmod(\Bun_G)_{\on{co}}\times \Dmod(\Bun_G) \to \Vect$$
denotes the pairing corresponding to the identification
$$\bD_{\Bun_G}:\Dmod(\Bun_G)^\vee\simeq \Dmod(\Bun_G)_{\on{co}}.$$

\medskip

On the one hand, by \corref{c:Eis op}, for $\CF_G\in \Dmod(\Bun_G)_{\on{co}}$, the condition that $\CF_G$ be right-orthogonal
to the essential image of $\Eis_{\on{co},*}$ for a given parabolic $P$ is equivalent to
$$\langle \Eis_!(\CF_M),\CF_G\rangle_{\Bun_G}=0,\quad \CF_M\in \Dmod(\Bun_M).$$ 

If $\CF_G= j_{\on{co},*}(\CF_U)$, then the above is equivalent to
$$\langle j^*\circ \Eis_!(\CF_M),\CF_U\rangle_U=0,$$
where
$$\langle-,-\rangle_{U}: \Dmod(U)_{\on{co}}\times \Dmod(U) \to \Vect$$
is the pairing corresponding to $\bD_U:\Dmod(U)^\vee\simeq \Dmod(U)$. 

\medskip

On the other hand, the condition that $\on{CT}_{\on{co},*}(\CF_G)=0$ for the same parabolic 
is equivalent to
$$\langle \Eis_*(\CF_M),\CF_G\rangle_{\Bun_G},$$
i.e.,
$$\langle j^*\circ \Eis_*(\CF_M),\CF_U\rangle_U=0.$$ 

\medskip

Hence, the assertion of \propref{p:cusp co via CT} follows from the next one, proved in \secref{ss:proof of tilde}:

\begin{prop}  \label{p:via BunP tilde} \hfill

\medskip

\noindent{\em(a)} 
For $\CF_M\in \Dmod(\Bun_M)$, the object $\Eis_*(\CF_M)$ admits an increasing filtration (indexed by a poset)
with subquotients of the form $\Eis_!(\CF^\alpha_M)$, $\CF^\alpha_M\in \Dmod(\Bun_M)$. 

\medskip

\noindent{\em(b)} Assume that $\CF_M$ is supported on finitely many connected components of $\Bun_M$, and
let $(U\overset{j}\hookrightarrow \Bun_G) \in \on{op-qc}(G)$. Then: 

\smallskip

\noindent{\em(i)} The objects $j^* \circ \Eis_!(\CF^\alpha_M)$ from point \emph{(a)} are zero for
all but finitely many $\alpha$'s. 

\smallskip

\noindent{\em(ii)} The object $j^* \circ \Eis_!(\CF_M)$ is a finite successive extension of objects of the form
$j^* \circ \Eis_*(\CF^\alpha_M)$, $\CF^\alpha_M\in \Dmod(\Bun_M)$. 

\end{prop} 

\end{proof}

\sssec{}

We now observe:

\begin{prop} \label{p:two cusp}
Let $\CF\in \Dmod(\Bun_G)_{\on{co}}$ be such that there exists
$(U_G\overset{j_G}\hookrightarrow \Bun_G) \in \on{op-qc}(G)$ such that the map
$$\CF\to (j_G)_{\on{co},*}\circ (j_G)^*_{\on{co}}(\CF)$$
is an isomorphism. Then $\psId_{\Bun_G,\on{naive}}(\CF)\in \Dmod(\Bun_G)_{\on{cusp}}$
if and only if $\on{CT}_{\on{co},*}(\CF)=0$ for all proper parabolics.
\end{prop}

\begin{proof}

We claim that for $\CF$ satisfying the condition of the proposition, for a given parabolic $P$,
$$\on{CT}_*\circ \psId_{\Bun_G,\on{naive}}(\CF)=0 \,\Leftrightarrow\, \on{CT}_{\on{co},*}(\CF)=0.$$

Indeed, the implication $\Leftarrow$ holds for \emph{any} $\CF$ by \propref{p:naive and functors}.

\medskip

Conversely, let $(U_M\overset{j_M}\hookrightarrow \Bun_M) \in \on{op-qc}(M)$ be as in \propref{p:CT co expl}. 
For $$\CF\simeq (j_G)_{\on{co},*}(\CF_{U_G}),$$ 
by \propref{p:CT co expl}, we have
\begin{multline*}
\on{CT}_{\on{co},*}(\CF)\simeq (j_M)_{\on{co},*}\circ 
(j_M)^*\circ \on{CT}_*\circ (j_G)_*(\CF_{U_G})\simeq \\
\simeq (j_M)_{\on{co},*}\circ (j_M)^*\circ \on{CT}_*\circ \psId_{\Bun_G,\on{naive}}\circ (j_G)_{\on{co},*}(\CF_{U_G})\simeq \\
\simeq (j_M)_{\on{co},*}\circ (j_M)^*\circ \on{CT}_*\circ\psId_{\Bun_G,\on{naive}}(\CF).
\end{multline*}

\end{proof}

\sssec{}

Combining Propositions \ref{p:!-supp cusp co}, \ref{p:cusp co via CT} and \ref{p:two cusp} we obtain:

\begin{cor} \label{c:descr cusp co} 
For $\CF\in \Dmod(\Bun_G)_{\on{co}}$ the following conditions are equivalent:

\smallskip

\noindent{\em(i)} $\CF\in \Dmod(\Bun_G)_{\on{co,cusp}}$;

\smallskip

\noindent{\em(ii)} There exists $(U\overset{j}\hookrightarrow \Bun_G) \in \on{op-qc}(G)$ such that the map
$\CF\to j_{\on{co},*}\circ j^*_{\on{co}}(\CF)$ is an isomorphism \emph{and}
$\psId_{\Bun_G,\on{naive}}(\CF)\in \Dmod(\Bun_G)_{\on{cusp}}$.

\smallskip

\noindent{\em(ii')} There exists $(U\overset{j}\hookrightarrow \Bun_G) \in \on{op-qc}(G)$ such that the map
$\CF\to j_{\on{co},*}\circ j^*_{\on{co}}(\CF)$ is an isomorphism \emph{and} $\on{CT}_{\on{co},*}(\CF)=0$ for all proper parabolics.

\smallskip

\noindent{\em(iii)} For $(\CU_G\overset{\jmath_G}\hookrightarrow \Bun_G) \in \on{op-qc}(G)$ as in \propref{p:!-supp cusp},
the map $\CF\to (\jmath_G)_{\on{co},*}\circ (\jmath_G)^*_{\on{co}}(\CF)$
is an isomorphism \emph{and}
$\psId_{\Bun_G,\on{naive}}(\CF)\in \Dmod(\Bun_G)_{\on{cusp}}$.

\smallskip

\noindent{\em(iii')} For $(\CU_G\overset{\jmath_G}\hookrightarrow \Bun_G) \in \on{op-qc}(G)$ as in \propref{p:!-supp cusp},
the map $\CF\to (\jmath_G)_{\on{co},*}\circ (\jmath_G)^*_{\on{co}}(\CF)$
is an isomorphism \emph{and} $\on{CT}_{\on{co},*}(\CF)=0$ for all proper parabolics.

\end{cor}

\sssec{Proof of \thmref{t:cusp naive equiv}}

From \corref{c:descr cusp co} we obtain that the functor $\psId_{\Bun_G,\on{naive}}$ sends
$$\Dmod(\Bun_G)_{\on{co,cusp}}\to \Dmod(\Bun_G)_{\on{cusp}}.$$

\medskip

We construct the inverse functor as follows. Let $(\CU_G\overset{\jmath_G}\hookrightarrow \Bun_G) \in \on{op-qc}(G)$
be as in \propref{p:!-supp cusp}. The sought-for functor
$$\Dmod(\Bun_G)_{\on{cusp}}\to \Dmod(\Bun_G)_{\on{co}}$$ is 
$$\CF\mapsto (\jmath_G)_{\on{co},*}\circ (\jmath_G)^*(\CF).$$

We claim that the image of this functor lands in $\Dmod(\Bun_G)_{\on{co,cusp}}$. Indeed, by \propref{p:two cusp}, its suffices
to check that
$$\psId_{\Bun_G,\on{naive}}\circ  (\jmath_G)_{\on{co},*}\circ (\jmath_G)^*(\CF)\in \Dmod(\Bun_G)_{\on{cusp}}.$$

However,
$$\psId_{\Bun_G,\on{naive}}\circ  (\jmath_G)_{\on{co},*}\circ (\jmath_G)^*(\CF)\simeq
(\jmath_G)_*\circ (\jmath_G)^*(\CF),$$
and the latter is isomorphic to $\CF$ by \propref{p:two cusp}.

\medskip

Let us now check that the two functors are inverses of each other. However, we have just shown that the composition
$$\Dmod(\Bun_G)_{\on{cusp}} \to \Dmod(\Bun_G)_{\on{co,cusp}}\to \Dmod(\Bun_G)_{\on{cusp}}$$
is isomorphic to the identity functor. 

\medskip

For the composition in the other direction, for $\CF\in \Dmod(\Bun_G)_{\on{co,cusp}}$ we consider
$$(\jmath_G)_{\on{co},*}\circ (\jmath_G)^*\circ \psId_{\Bun_G,\on{naive}}(\CF),$$
which by \corref{c:naive and res} is isomorphic to
$$(\jmath_G)_{\on{co},*}\circ (\jmath_G)^*_{\on{co}}(\CF),$$
and the latter is isomorphic to $\CF$ by \propref{p:!-supp cusp co}. 

\qed

\ssec{Proof of \propref{p:via BunP tilde}}  \label{ss:proof of tilde}

\sssec{}

The proof of the proposition uses the relative compactification $\Bun_P\overset{r}\hookrightarrow \wt\Bun_P$
of the map $\sfp$, introduced in \cite[Sect. 1.3.6]{BG}:

\begin{gather}  \label{e:compact}
\xy
(-20,0)*+{\Bun_G}="X";
(20,0)*+{\Bun_M}="Y";
(0,20)*+{\wt{\Bun}_P}="Z";
(-20,20)*+{\Bun_P}="W";
{\ar@{->}_{\wt{\sfp}} "Z";"X"};
{\ar@{->}^{\wt\sfq} "Z";"Y"};
{\ar@{^{(}->}^r "W";"Z"};  
\endxy
\end{gather}

\medskip

Note that for $\CF_M\in \Dmod(\Bun_M)$, we have
$$\Eis_*(\CF_M)\simeq \wt\sfp_*\left(\wt\sfq^!(\CF_M)\sotimes r_*(\omega_{\Bun_P})\right)\simeq
\wt\sfp_!\left(\wt\sfq^!(\CF_M)\sotimes r_*(\omega_{\Bun_P})\right),$$
the latter isomorphism due to the fact that $\wt\sfp$ is proper.  Here the notation $\sotimes$ (and, in the sequel, $\overset{*}\otimes$)
follows \cite[Sect. 1.1.5]{DrGa3}. 

\medskip

Recall now that according to \cite[Theorem 5.1.5]{BG}, the object
$$r_*(\omega_{\Bun_P})\in \Dmod(\wt\Bun_P)$$
is \emph{universally locally acyclic} (a.k.a. ULA)\footnote{See \cite[Sect. 1.1.5]{DrGa3} for what the ULA condition means.}
with respect to the map $\wt\sfq$. This implies
that 
$$\wt\sfq^!(\CF_M)\sotimes r_*(\omega_{\Bun_P})\simeq \wt\sfq^*(\CF_M)\overset{*}\otimes r_*(\omega_{\Bun_P})[-2\dim(\Bun_M)].$$

Thus, we obtain that, up to a cohomological shift, $\Eis_*(\CF_M)$ is isomorphic to 

\begin{equation} \label{e:exp for Eis* via comp}
\wt\sfp_!\left(\wt\sfq^*(\CF_M)\overset{*}\otimes r_*(\omega_{\Bun_P})\right).
\end{equation}

\sssec{}

Let $\Lambda_{G,P}^{\on{pos}}$ be the monoid of linear combinations
$$\theta=\underset{i}\Sigma\, n_i\cdot \alpha_i,$$
where $n_i\in \BZ^{\geq 0}$ and $\alpha_i$ is a simple coroot of $G$, which is not in $M$. 

\medskip

For each $\theta$, we let $\on{Mod}^{\theta,+}_{\Bun_M}$ be a version of the Hecke stack, introduced in 
\cite[Sect. 3.1]{BFGM}:
$$
\xy
(-20,0)*+{\Bun_M}="X";
(20,0)*+{\Bun_M.}="Y";
(0,20)*+{\on{Mod}^{\theta,+}_{\Bun_M}}="Z";
{\ar@{->}_{\hl} "Z";"X"};
{\ar@{->}^{\hr} "Z";"Y"};
\endxy
$$

Set
$$\on{Mod}^{\theta,+}_{\Bun_P}:=\Bun_P\underset{\Bun_M}\times \on{Mod}^{\theta,+}_{\Bun_M},$$
where the fiber product is formed using the map $\hl:\on{Mod}^{\theta,+}_{\Bun_M}\to \Bun_M$. 

\medskip 

According to \cite[Proposition 6.2.5]{BG}, there is a canonically defined locally closed embedding
$$r^\theta:\on{Mod}^{\theta,+}_{\Bun_P}\to \wt\Bun_P,$$ making the following diagram commute
$$
\xy
(-20,0)*+{\Bun_M}="X";
(20,0)*+{\Bun_M.}="Y";
(0,20)*+{\on{Mod}^{\theta,+}_{\Bun_M}}="Z";
(-20,30)*+{\Bun_P}="W";
(0,50)*+{\on{Mod}^{\theta,+}_{\Bun_P}}="U";
(20,40)*+{\wt\Bun_P}="V";
{\ar@{->}_{\hl} "Z";"X"};
{\ar@{->}^{\hr} "Z";"Y"};
{\ar@{->}_{\sfq} "W";"X"};
{\ar@{->}_{'\sfq} "U";"Z"};
{\ar@{->}_{'\hl} "U";"W"};
{\ar@{->}^{\wt\sfq} "V";"Y"};
{\ar@{->}^{r^\theta} "U";"V"};
\endxy
$$
(The right diamond is intentionally lopsided to emphasize that it is \emph{not} Cartesian.) 

\medskip

Furthermore, 
\begin{equation} \label{e:decomp tilde}
\wt\Bun_P=\underset{\theta\in \Lambda_{G,P}^{\on{pos}}}\bigsqcup\, r^\theta(\on{Mod}^{\theta,+}_{\Bun_P}).
\end{equation}

For $\theta=0$, the map $\hl$ is an isomorphism, and the resulting map
$$\Bun_P\simeq \on{Mod}^{0,+}_{\Bun_P}\overset{r^0}\hookrightarrow \wt\Bun_P$$
is the map $r$ in \eqref{e:compact}.

\medskip

The following is easy to see from the construction:

\begin{lem} \label{l:fin theta}
For  $(U\overset{j}\hookrightarrow \Bun_G) \in \on{op-qc}(G)$ and $\mu\in \pi_1(M)$, the preimage of $U\times \Bun_M^\mu$ under the map
$$\on{Mod}^{\theta,+}_{\Bun_P}\overset{r^\theta}\hookrightarrow \wt\Bun_P\overset{\wt\sfp\times \wt\sfq}\longrightarrow
\Bun_G\times \Bun_M$$
is empty for all but finitely many elements $\theta$.
\end{lem}

\sssec{}  \label{sss:fin theta}

The decomposition \eqref{e:decomp tilde} endows the object 
$$r_*(\omega_{\Bun_P})\in \Dmod(\wt\Bun_P)$$
with an increasing filtration, indexed by the poset $\Lambda_{G,P}^{\on{pos}}$, with the $\theta$ subquotient
equal to
$$(r^\theta)_!\circ (r^\theta)^*\circ r_*(\omega_{\Bun_P}).$$

\medskip

Hence, by the projection formula, the object in \eqref{e:exp for Eis* via comp} admits a filtration, indexed
by $\Lambda_{G,P}^{\on{pos}}$, with the $\theta$ subquotient equal to
\begin{equation} \label{e:expr for subquotient}
\wt\sfp_!\circ r^\theta_!\left((r^\theta)^*\circ \wt\sfq^*(\CF_M)\overset{*}\otimes (r^\theta)^*\circ r_*(\omega_{\Bun_P})\right).
\end{equation}

Moreover, if $\CF_M$ is supported on finitely many components of $\Bun_M$, 
the restriction of the subquotient \eqref{e:expr for subquotient} to $U\in \on{op-qc}(G)$ is zero for all but finitely many $\theta$
by \lemref{l:fin theta}. 

\sssec{}

We have the following assertion, proved by the same argument as \cite[Theorem 6.2.10]{BG}:

\begin{lem}  \label{l:on stratum}
The object $(r^\theta)^*\circ r_*(\omega_{\Bun_P})\in \Dmod(\on{Mod}^{\theta,+}_{\Bun_P})$ is lisse
when !-restricted to the fiber of the map 
$$'\sfq:\on{Mod}^{\theta,+}_{\Bun_P}\to \on{Mod}^{\theta,+}_{\Bun_M}$$
over any $k$-point of $\on{Mod}^{\theta,+}_{\Bun_M}$. 
\end{lem}

\begin{cor} \label{c:on stratum}
$(r^\theta)^*\circ r_*(\omega_{\Bun_P})\simeq {}'\sfq^*(\CK^\theta)$ for 
some $\CK^\theta\in \Dmod(\on{Mod}^{\theta,+}_{\Bun_M})$.
\end{cor}

\begin{proof}

Follows from \lemref{l:on stratum} plus the combination of the following three facts: (1) the map $'\sfq$ is smooth;
(2) $(r^\theta)^*\circ r_*(\omega_{\Bun_P})$ is holonomic with regular singularities; (3) the fibers of the map
$'\sfq$ are contractible (and hence any RS local system on such a fiber is canonically trivial). 

\end{proof}

\sssec{}
 
By \corref{c:on stratum}, we can rewrite the subquotient \eqref{e:expr for subquotient} as
$$\wt\sfp_!\circ r^\theta_!\left((r^\theta)^*\circ \wt\sfq^*(\CF_M)\overset{*}\otimes ({}'\sfq)^*(\CK^\theta)\right),$$
and further, using the fact that
$$\wt\sfq\circ r^\theta=\hr\circ {}'\sfq \text{ and } \wt\sfp\circ r^\theta=\sfp\circ {}'\hl$$
as
$$\sfp_!\circ {}'\hl_!\circ  {}'\sfq^*\left(\hr^*(\CF_M)\overset{*}\otimes \CK^\theta\right)\simeq
\sfp_!\circ \sfq^* \left(\hl_!(\hr^*(\CF_M)\overset{*}\otimes \CK^\theta)\right).$$

\medskip

To summarize, we identify the subquotient \eqref{e:expr for subquotient} with
$$\Eis_!\left(\hl_!(\hr^*(\CF_M)\overset{*}\otimes \CK^\theta)\right),$$
as required in \propref{p:via BunP tilde}(a). The finiteness assertion in \propref{p:via BunP tilde}(b)(i) follows
from the finiteness at the end of \secref{sss:fin theta}. 

\sssec{}

The proof of \propref{p:via BunP tilde}(b)(ii) is similar, but with the following modification: 

\medskip

Let $k_{\Bun_P}\in  \Dmod(\Bun_P)$ be the ``constant sheaf" D-module, i.e., the Verdier dual of $\omega_{\Bun_P}$.

\medskip

Then the object
$$r_!(k_{\Bun_P})\in \Dmod(\wt\Bun_P)$$
admits a \emph{decreasing} filtration, indexed by the poset $\Lambda_{G,P}^{\on{pos}}$, with 
the $\theta$ subquotient being
$$(r^\theta)_*\circ (r^\theta)^!\circ r_!(k_{\Bun_P}).$$

However, this filtration is finite on the preimage of $U\times \Bun_M^\mu$
for any $U\in \on{op-qc}(G)$ and $\mu\in \pi_1(M)$ under the map $\wt\sfp\times \wt\sfq$,
again by \lemref{l:fin theta}. 

\section{Interaction with the genuine pseudo-identity functor}  \label{s:pseudo}

\ssec{The pseudo-identity functor}  \label{ss:pseudo}

\sssec{}

We now recall that in \cite[Sect. 4.4.8]{DrGa2} another functor, denoted
$$\psId_{\Bun_G,!}:\Dmod(\Bun_G)_{\on{co}}\to \Dmod(\Bun_G)$$
was introduced.

\medskip

Namely, in terms of the equivalences
\begin{multline}  \label{e:functors}
\on{Funct}_{\on{cont}}(\Dmod(\Bun_G)_{\on{co}},\Dmod(\Bun_G))\simeq
(\Dmod(\Bun_G)_{\on{co}})^\vee\otimes \Dmod(\Bun_G)\simeq  \\
\simeq \Dmod(\Bun_G)\otimes \Dmod(\Bun_G)\simeq \Dmod(\Bun_G\times \Bun_G),
\end{multline}
the functor
$\psId_{\Bun_G,!}$ corresponds to the object
$$(\Delta_{\Bun_G})_!(k_{\Bun_G})\in \Dmod(\Bun_G\times \Bun_G).$$

\sssec{}

Note the following feature of the functor $\psId_{\Bun_G,!}$, parallel to one for $\psId_{\Bun_G,\on{naive}}$,
given by \lemref{l:naive self-dual}. 

\begin{lem}
Under the identification $\Dmod(\Bun_G)^\vee\simeq \Dmod(\Bun_G)_{\on{co}}$, we have
$$(\psId_{\Bun_G,!})^\vee \simeq \psId_{\Bun_G,!}.$$
\end{lem}

\begin{proof}
This is just the fact that the object $(\Delta_{\Bun_G})_!(k_{\Bun_G})\in \Dmod(\Bun_G\times \Bun_G)$
is equivariant with respect to the flip.
\end{proof}

\sssec{}

The goal of this section and the next is to prove:

\begin{thm} \label{t:duality}
The functor $\psId_{\Bun_G,!}$ is an equivalence.
\end{thm}

The proof will rely on a certain geometric result, namely, \propref{p:diff} which will be proved in \cite{Sch}.

\ssec{Relation between the two functors}

\sssec{}

Consider again the map 
$$\Delta_{\Bun_G}:\Bun_G\to \Bun_G\times \Bun_G.$$

It naturally factors as
$$\Bun_G\overset{\on{id}_{\Bun_G}^Z}\to \Bun_G\times B(Z_G)\overset{\Delta^Z_{\Bun_G}}\longrightarrow \Bun_G\times \Bun_G,$$
where:

\begin{itemize}

\item$Z_G$ denotes the center of $G$, and $B(Z_G)$ is its classifying stack; 

\medskip

\item The map $\on{id}^Z_{\Bun_G}$ is given by the identity map $\Bun_G\to \Bun_G$,
and $$\Bun_G\to \on{pt}\overset{\on{triv}}\to B(Z_G),$$ where $\on{triv}:\on{pt}\to B(Z_G)$ corresponds to the trivial $Z_G$-bundle;

\medskip

\item The composition $\on{pr}_1\circ \Delta^Z_{\Bun_G}$ is projection on the first
factor $\Bun_G\times B(Z_G)\to \Bun_G$;

\medskip

\item The composition $\on{pr}_2\circ \Delta^Z_{\Bun_G}$ is given by the natural action of
$B(Z_G)$ on $\Bun_G$.

\end{itemize}

\begin{rem}
Note that if $G$ is a torus, the map $\Delta^Z_{\Bun_G}$ is an isomorphism. 
\end{rem}

\sssec{}

We write
$$(\Delta_{\Bun_G})_*(\omega_{\Bun_G}) \simeq (\Delta^Z_{\Bun_G})_*\circ (\on{id}^Z_{\Bun_G})_*(\omega_{\Bun_G}).$$

In addition, 
$$(\Delta_{\Bun_G})_!(k_{\Bun_G}) \simeq (\Delta^Z_{\Bun_G})_!\circ (\on{id}^Z_{\Bun_G})_!(k_{\Bun_G})\simeq
(\Delta^Z_{\Bun_G})_!\circ (\on{id}^Z_{\Bun_G})_!(\omega_{\Bun_G})[-2\dim(\Bun_G)],$$
the latter isomorphism is due to the fact that $\Bun_G$ is smooth.

\medskip

It is easy to see that
$$\on{triv}_!(k)\simeq \on{triv}_*(k)[-\dim(Z_G)].$$

Hence, 
$$
(\Delta_{\Bun_G})_!(k_{\Bun_G}) 
\simeq (\Delta^Z_{\Bun_G})_!\circ (\on{id}^Z_{\Bun_G})_*(\omega_{\Bun_G})[-2\dim(\Bun_G)-\dim(Z_G)].$$

\medskip

Now, the morphism 
$$\Delta^Z_{\Bun_G}: \Bun_G\times B(Z_G)\to \Bun_G\times \Bun_G$$
is schematic and separated. Hence, we obtain a natural transformation 
\begin{equation} \label{e:! to * diag}
(\Delta^Z_{\Bun_G})_! \to (\Delta^Z_{\Bun_G})_*.
\end{equation}

Summarizing, we obtain a map
\begin{equation} \label{e:map between kernels}
(\Delta_{\Bun_G})_!(k_{\Bun_G})\to (\Delta_{\Bun_G})_*(\omega_{\Bun_G})[-2\dim(\Bun_G)-\dim(Z_G)].
\end{equation} 

\sssec{}

From \eqref{e:map between kernels} and \secref{sss:naive by kernel}, we obtain a natural transformation:

\begin{equation} \label{e:diff}
\psId_{\Bun_G,!}\to \psId_{\Bun_G,\on{naive}}[-2\dim(\Bun_G)-\dim(Z_G)]
\end{equation} 
as functors $\Dmod(\Bun_G)_{\on{co}} \to \Dmod(\Bun_G)$. 

\medskip

Let $$\psId_{\Bun_G,\on{diff}}: \Dmod(\Bun_G)_{\on{co}} \to \Dmod(\Bun_G)$$ denote the cone of
the natural transformation \eqref{e:diff}.

\sssec{}

We claim:

\begin{prop} \label{p:diff}
The functor $\psId_{\Bun_G,\on{diff}}$ admits a decreasing filtration, indexed by a poset, with subquotients 
being functors of the form 
\begin{multline*}
\Dmod(\Bun_G)_{\on{co}} \overset{\on{CT}^\mu_{\on{co},*}}\longrightarrow 
\Dmod(\Bun^\mu_M)_{\on{co}}\overset{\psId_{\Bun^\mu_M,\on{naive}}}\longrightarrow \Dmod(\Bun^\mu_M) 
\overset{\sF^{\mu,\mu'}} \longrightarrow  \\
\to \Dmod(\Bun^{\mu'}_M) \overset{\Eis^{\mu',-}_*}\longrightarrow \Dmod(\Bun_G),
\end{multline*}
for a \emph{proper} parabolic $P$ with Levi quotient $M$, where $\mu,\mu'\in \pi_1(M)$ and 
$\sF^{\mu,\mu'}$ is \emph{some} functor $\Dmod(\Bun^\mu_M)\to \Dmod(\Bun^{\mu'}_M)$.
Furthermore, for a pair 
$$(U_1\overset{j_1}\hookrightarrow \Bun_G),\,\, (U_2\overset{j_2}\hookrightarrow \Bun_G)\in  \on{op-qc}(G),$$ 
the induced filtration on 
$$j_1^*\circ \psId_{\Bun_G,\on{diff}} \circ (j_2)_{\on{co},*}$$
is finite. 
\end{prop}

The proof of \propref{p:diff} is analogous to that of \propref{p:via BunP tilde} and will be given in \cite{Sch}. 

\medskip

As its
geometric ingredient, instead of the stack $\wt\Bun_P$ appearing in the proof of \propref{p:via BunP tilde}, one
uses a compactification of the morphism $\Delta_{\Bun_G}^{Z_G}$ which can be constructed using Vinberg's
canonical semi-group of \cite{Vi} attached to $G$. 

\ssec{Pseudo-identity and cuspidality}

\sssec{}

As a consequence of \propref{p:diff}, we obtain:

\begin{cor}  \label{c:! and * on cusp}
The morphism \eqref{e:diff} induces an isomorphism
$$\psId_{\Bun_G,!}|_{\Dmod(\Bun_G)_{\on{co,cusp}}}\simeq \psId_{\Bun_G,\on{naive}}|_{\Dmod(\Bun_G)_{\on{co,cusp}}}[-2\dim(\Bun_G)-\dim(Z_G)].$$
\end{cor}

\begin{proof}

By the definition of $\Dmod(\Bun_G)$, it is sufficient to show that for any 
$(U_1\overset{j_1}\hookrightarrow \Bun_G)\in  \on{op-qc}(G)$, the map \eqref{e:diff} induces an isomorphism
\begin{multline*}
j_1^*\circ \psId_{\Bun_G,!}|_{\Dmod(\Bun_G)_{\on{co,cusp}}}\to \\
\to j_1^*\circ  \psId_{\Bun_G,\on{naive}}|_{\Dmod(\Bun_G)_{\on{co,cusp}}}[-2\dim(\Bun_G)-\dim(Z_G)].
\end{multline*}

\medskip

Let us take $U_2:=\CU_G$ as in \propref{p:!-supp cusp}. By \propref{p:!-supp cusp co}, it suffices to show that for
$\CF\in \Dmod(\Bun_G)_{\on{co,cusp}}$
$$j_1^*\circ \psId_{\Bun_G,\on{diff}}\circ (j_2)_{\on{co},*}\circ (j_2)^*_{\on{co}}(\CF)=0.$$

\medskip

However, this follows from \propref{p:diff}: 

\medskip

Indeed, the object in question has a finite filtration, with subquotients isomorphic to
$$j_1^*\circ \Eis^{\mu',-}_*\circ \sF^{\mu,\mu'}\circ \psId_{\Bun^\mu_M,\on{naive}}\circ \on{CT}^\mu_{\on{co},*}\circ (j_2)_{\on{co},*}\circ (j_2)^*_{\on{co}}(\CF),$$
which, by \propref{p:!-supp cusp co},  is isomorphic to 
$$j_1^*\circ \Eis^{\mu',-}_*\circ \sF^{\mu,\mu'}\circ \psId_{\Bun^\mu_M,\on{naive}}(\on{CT}^\mu_{\on{co},*}(\CF)),$$
while $\on{CT}^\mu_{\on{co},*}(\CF)=0$ by \corref{c:descr cusp co}.

\end{proof}

\begin{cor} \label{c:equiv on cusp}
The functor $\psId_{\Bun_G,!}$ induces an equivalence
$$\Dmod(\Bun_G)_{\on{co,cusp}}\to \Dmod(\Bun_G)_{\on{cusp}}.$$
\end{cor}

\begin{proof}

Follows from \thmref{t:cusp naive equiv} and \corref{c:! and * on cusp}.

\end{proof}

\sssec{}

The next assertion is a crucial step in the proof of \thmref{t:duality}:

\begin{prop}  \label{p:Hom out of cusp}
The functor $\psId_{\Bun_G,!}$ induces an isomorphism
$$\Hom_{\Dmod(\Bun_G)_{\on{co}}}(\CF',\CF)\to
\Hom_{\Dmod(\Bun_G)}(\psId_{\Bun_G,!}(\CF'),\psId_{\Bun_G,!}(\CF)),$$
provided that $\CF'\in \Dmod(\Bun_G)_{\on{co,cusp}}$. 
\end{prop}

\ssec{Proof of \propref{p:Hom out of cusp}}

\sssec{}  \label{sss:from open}

Let us first assume that $\CF$ has the form $j_{\on{co},*}(\CF_U)$ for some
$(U\overset{j}\hookrightarrow \Bun_G) \in \on{op-qc}(G)$. 

\medskip

Consider the commutative diagram

\medskip

\begin{equation} \label{e:Hom cusp}
\CD
\Hom(\CF',\CF)  @>>>    \Hom(\psId_{\Bun_G,\on{naive}}(\CF'),\psId_{\Bun_G,\on{naive}}(\CF))  \\
@VVV   @VVV  \\
\Hom(\psId_{\Bun_G,!}(\CF'),\psId_{\Bun_G,!}(\CF)) @>>>
\Hom(\psId_{\Bun_G,!}(\CF'),\psId_{\Bun_G,\on{naive}}(\CF)[d]),
\endCD
\end{equation} 
where $d=-2\dim(\Bun_G)-\dim(Z_G)$. 

\medskip

We need to show that the left vertical arrow is an isomorphism. We will do so by showing that all
the other arrows are isomorphisms. 

\sssec{}

First, we claim that upper horizontal arrow in \eqref{e:Hom cusp}
 is an isomorphism for any $\CF'\in \Dmod(\Bun_G)_{\on{co}}$ and
$\CF=j_*(\CF_U)$. Indeed, the map in question fits into a commutative diagram
$$
\CD
\Hom_{\Dmod(\Bun_G)_{\on{co}}}(\CF',j_{\on{co},*}(\CF_U))  @>>>  
 \Hom(\psId_{\Bun_G,\on{naive}}(\CF'),\psId_{\Bun_G,\on{naive}}\circ j_{\on{co},*}(\CF_U))   \\
 @V{\sim}VV     @VV{\sim}V   \\  
 \Hom_{\Dmod(U)}(j_{\on{co}}^*(\CF'),\CF_U)   & & 
 \Hom_{\Dmod(\Bun_G)}(\psId_{\Bun_G,\on{naive}}(\CF'),j_*(\CF_U)) \\
 @V{\on{id}}VV    @VV{\sim}V  \\
\Hom_{\Dmod(U)}(j_{\on{co}}^*(\CF'),\CF_U) @>{\sim}>> \Hom_{\Dmod(U)}(j^*\circ \psId_{\Bun_G,\on{naive}}(\CF'),\CF_U). \\
\endCD
$$

\sssec{}

The right vertical arrow in \eqref{e:Hom cusp} is an isomorphism by \corref{c:! and * on cusp}. 

\medskip

To show that the lower horizontal arrow is an isomorphism, using \corref{c:equiv on cusp}, it suffices to
show that for any $\CF''\in \Dmod(\Bun_G)_{\on{cusp}}$, we have
$$\Hom_{\Dmod(\Bun_G)}(\CF'',\psId_{\Bun_G,\on{diff}}\circ j_{\on{co},*}(\CF_U))=0.$$

\medskip

By \propref{p:!-supp cusp},
\begin{multline*}
\Hom_{\Dmod(\Bun_G)}(\CF'', \psId_{\Bun_G,\on{diff}}\circ j_{\on{co},*}(\CF_U))\simeq  \\
\simeq \Hom_{\Dmod(\CU_G)}(\jmath_G^*(\CF''), \jmath_G^*\circ \psId_{\Bun_G,\on{diff}}\circ j_{\on{co},*}(\CF_U)).
\end{multline*}

Applying \propref{p:diff}, we obtain that it suffices to show that for $\CF''\in \Dmod(\Bun_G)_{\on{cusp}}$
$$\Hom_{\Dmod(\CU_G)}\left(\jmath_G^*(\CF''),\jmath_G^*\circ 
\Eis^{\mu',-}_*\circ \sF^{\mu,\mu'}\circ \psId_{\Bun^\mu_M,\on{naive}}\circ \on{CT}^\mu_{\on{co},*}\circ j_{\on{co},*}(\CF_U)\right)=0,$$
which by \propref{p:!-supp cusp} is equivalent to 
$$\Hom_{\Dmod(\Bun_G)}\left(\CF'',
\Eis^{\mu',-}_*\circ \sF^{\mu,\mu'}\circ \psId_{\Bun^\mu_M,\on{naive}}\circ \on{CT}^\mu_{\on{co},*}\circ j_{\on{co},*}(\CF_U)\right)=0.$$

Now, $\Dmod(\Bun_G)_{\on{cusp}}$ is \emph{left-orthogonal} to the essential image of $\Eis^{\mu',-}_*$
by \thmref{t:weird adj}, implying the desired vanishing. 

\sssec{}

We will now reduce the assertion of \propref{p:Hom out of cusp} to the situation of \secref{sss:from open}.

\medskip

Let us recall that according to \cite[Theorem 4.1.8]{DrGa2}, any element
$(U\overset{j}\hookrightarrow \Bun_G) \in \on{op-qc}(G)$ is contained in one
which is \emph{co-truncative}. See \cite[Sect. 3.8]{DrGa2} for what it means for an open substack
to be co-truncative. In particular, the open substack $\CU_G$ of \propref{p:!-supp cusp} can be 
enlarged so that it is co-trunactive.

\medskip

Recall also that for a co-truncative open substack $U\overset{j}\hookrightarrow \Bun_G$, the functor
$j_{\on{co},*}$ has a (continuous) right adjoint, denoted $j^?$, see \cite[Sect. 4.3]{DrGa2}.

\medskip

Any $\CF\in  \Dmod(\Bun_G)_{\on{co}}$ fits into an exact triangle
$$\CF_1\to \CF\to j_{\on{co},*}\circ j^?(\CF),$$
where $j^?(\CF_1)=0$ by \lemref{l:j ff}.

\medskip

We take $U$ to contain the substack $\CU_G$ as in \propref{p:!-supp cusp}, and assume that it is co-truncative. 
In view of \propref{p:!-supp cusp co} and \corref{c:equiv on cusp}, it remains to show that if $j^?(\CF)=0$, 
then
$$\Hom_{\Dmod(\Bun_G)}(\CF'',\psId_{\Bun_G,!}(\CF))=0, \quad \CF''\in \Dmod(\Bun_G)_{\on{cusp}}.$$

By \propref{p:!-supp cusp}, it suffices to show that 
$$j^?(\CF)=0\, \Rightarrow \, j^*\circ \psId_{\Bun_G,!}(\CF)=0.$$

However, this follows from (the nearly tautological) \cite[Corollary 6.6.3]{Ga2}.

\qed

\section{The strange functional equation and proof of the equivalence}  \label{s:strange}

In this section we will carry out the two main tasks of this paper: we will prove the strange functional
equation (\thmref{t:strange} below) and finish the proof of \thmref{t:duality} (that says that
the functor $\psId_{\Bun_G,!}$ is an equivalence). 

\ssec{The strange functional equation}

In this subsection we will study the behavior of the functor $\psId_{\Bun_G,!}$ on the subcategory 
$$\Dmod(\Bun_G)_{\on{co,Eis}}\subset \Dmod(\Bun_G)_{\on{co}}.$$

\sssec{}

First, we have the following ``strange" result:

\begin{thm} \label{t:strange}
For a parabolic $P$ and its opposite $P^-$ we have a canonical
isomorphism of functors
$$\Eis_!\circ \psId_{\Bun_M,!}\simeq \psId_{\Bun_G,!}\circ \Eis^-_{\on{co},*}.$$
\end{thm}

\begin{proof}

Both sides are continuous functors
$$\Dmod(\Bun_M)_{\on{co},*}\to \Dmod(\Bun_G),$$
that correspond to objects of
$$\Dmod(\Bun_M\times \Bun_G)$$
under the identification
\begin{multline*}
\on{Funct}_{\on{cont}}(\Dmod(\Bun_M)_{\on{co}},\Dmod(\Bun_G))\simeq
(\Dmod(\Bun_M)_{\on{co}})^\vee\otimes \Dmod(\Bun_G)\simeq  \\
\simeq \Dmod(\Bun_M)\otimes \Dmod(\Bun_G)\simeq \Dmod(\Bun_M\times \Bun_G),
\end{multline*}

We claim that both objects identify canonically with
$$((\sfq\times \sfp)\circ \Delta_{\Bun_P})_!(k_{\Bun_P}),$$
where the map in the formula is the same as
$$\Bun_P\overset{\sfq\times \sfp}\longrightarrow \Bun_M\times \Bun_G.$$

\medskip

The functor $\Eis_!\circ \psId_{\Bun_M,!}$ corresponds to the object, obtained by applying the functor
$$(\on{Id}_{\Dmod(\Bun_M)}\otimes \Eis_!):\Dmod(\Bun_M)\otimes \Dmod(\Bun_M)\to
\Dmod(\Bun_M)\otimes \Dmod(\Bun_G)$$
to 
$$(\Delta_{\Bun_M})_!(k_{\Bun_M})\in \Dmod(\Bun_M\times \Bun_M)\simeq \Dmod(\Bun_M)\otimes \Dmod(\Bun_M).$$

The functor $\on{Id}_{\Dmod(\Bun_M)}\otimes \Eis_!$ is left adjoint to the functor
$$\on{Id}_{\Dmod(\Bun_M)}\otimes \on{CT}_*\simeq (\on{id}_{\Bun_M}\times \sfq)_*\circ (\on{id}_{\Bun_M}\times \sfp)^!,$$
and hence is the !-Eisenstein series functor for the group $M\times G$ with respect to the parabolic $M\times P$.
I.e., it is given by
$$(\on{id}_{\Bun_M}\times \sfp)_!\times (\on{id}_{\Bun_M}\times \sfq)^*,$$
when applied to holonomic objects. 

\medskip

Base change along the diagram
$$
\CD
\Bun_P  @>{\Gamma_\sfq}>>  \Bun_M\times \Bun_P  @>{\on{id}_{\Bun_M}\times \sfp}>> \Bun_M\times \Bun_G \\
@V{\sfq}VV   @VV{\on{id}_{\Bun_M}\times \sfq}V   \\
\Bun_M  @>{\Delta_{\Bun_M}}>> \Bun_M\times \Bun_M
\endCD
$$
shows that
$$(\on{id}_{\Bun_M}\times \sfp)_!\times (\on{id}_{\Bun_M}\times \sfq)^*\circ (\Delta_{\Bun_M})_!(k_{\Bun_M})
\simeq ((\sfq\times \sfp)\circ \Delta_{\Bun_P})_!(k_{\Bun_P}),$$
as required. 
 
\medskip
 
The functor $\psId_{\Bun_G,!}\circ \Eis^-_{\on{co},*}$ corresponds to the object, obtained by applying the functor
\begin{multline*}
\left((\Eis^-_{\on{co},*})^\vee\otimes \on{Id}_{\Dmod(\Bun_G)}\right): \\
\Dmod(\Bun_G)\otimes \Dmod(\Bun_G)\to
\Dmod(\Bun_M)\otimes \Dmod(\Bun_G)
\end{multline*}
to the object 
$$(\Delta_{\Bun_G})_!(k_{\Bun_G})\in \Dmod(\Bun_G\times \Bun_G)\simeq \Dmod(\Bun_G)\otimes \Dmod(\Bun_G).$$

\medskip

We have:
$$(\Eis^-_{\on{co},*})^\vee\simeq \on{CT}^-_*,$$
and we recall that by \thmref{t:weird adj}
$$\on{CT}^-_*\simeq \on{CT}_!:=\underset{\mu}\bigoplus\, \on{CT}_!^\mu,$$
where $\on{CT}_!^\mu$ is the left adjoint of $\Eis^\mu_*$.

\medskip

Since $\on{CT}_!^\mu$ is the left adjoint of $\Eis_*^\mu$, we obtain that
$\on{CT}_!^\mu \otimes \on{Id}_{\Dmod(\Bun_G)}$ is the left adjoint of
$\Eis_*^\mu\otimes \on{Id}_{\Dmod(\Bun_G)}$, i.e., is the !-constant term functor 
for the group $G\times G$ with respect to the parabolic $P\times G$. Hence,
$$\on{CT}_!^\mu \otimes \on{Id}_{\Dmod(\Bun_G)}\simeq 
(\sfq^\mu\times \on{id}_{\Bun_G})_!\circ (\sfp^\mu\times \on{id}_{\Bun_G})^*,$$
when appied to holonomic objects
(the superscipt $\mu$ indicates that we are taking only the $\mu$-connected component
of $\Bun_P$).

\medskip

Taking the direct sum over $\mu$, we thus obtain
$$(\Eis^-_{\on{co},*})^\vee\otimes \on{Id}_{\Dmod(\Bun_G)}\simeq
(\sfq\times \on{id}_{\Bun_G})_!\circ (\sfp\times \on{id}_{\Bun_G})^*,$$
when applied to holonomic objects.

\medskip

Now, base change along the diagram 
$$
\CD
\Bun_P  @>{\Gamma_{\sfp}}>>  \Bun_P \times \Bun_G @>{\sfq\times \on{id}_{\Bun_G}}>> \Bun_M\times \Bun_G \\
@V{\sfp}VV  @VV{\sfp\times \on{id}_{\Bun_G}}V   \\
\Bun_G @>{\Delta_{\Bun_G}}>>  \Bun_G\times \Bun_G,
\endCD
$$
shows that 
$$
(\sfq\times \on{id}_{\Bun_G})_!\circ (\sfp\times \on{id}_{\Bun_G})^*\circ (\Delta_{\Bun_G})_!(k_{\Bun_G})
\simeq ((\sfq\times \sfp)\circ \Delta_{\Bun_P})_!(k_{\Bun_P}),
$$
as required. 

\end{proof}

\sssec{}

By passing to dual functors in the isomorphism
\begin{equation} \label{e:weird adj +}
\Eis_!\circ \psId_{\Bun_M,!}\simeq \psId_{\Bun_G,!}\circ \Eis^-_{\on{co},*}
\end{equation}
of \thmref{t:strange}, we obtain:

\begin{cor}  \label{c:CT and Psi}
There is a canonical isomorphism
\begin{equation} \label{e:CT and Psi 1}
\psId_{\Bun_M,!}\circ \on{CT}_{\on{co},?}\simeq \on{CT}^-_*\circ \psId_{\Bun_G,!}.
\end{equation}
\end{cor} 

\sssec{}

Consider now the commutative diagram:
\begin{equation} \label{e:weird adj -}
\CD
\Dmod(\Bun_G)_{\on{co}}  @>{\psId_{\Bun_G,!}}>>  \Dmod(\Bun_G)  \\
@A{\Eis_{\on{co},*}}AA   @AA{\Eis^-_!}A  \\
\Dmod(\Bun_M)_{\on{co}}  @>{\psId_{\Bun_M,!}}>>  \Dmod(\Bun_M).
\endCD
\end{equation}

By passing to the right adjoint functors along the vertical arrows, we obtain a natural transformation
\begin{equation} \label{e:CT and Psi 2}
\psId_{\Bun_M,!}\circ \on{CT}_{\on{co},?}\to \on{CT}^-_*\circ \psId_{\Bun_G,!}.
\end{equation} 

We now claim:

\begin{prop} \label{p:two maps}
The map \eqref{e:CT and Psi 2} equals the map \eqref{e:CT and Psi 1}, and, in particular, is an isomorphism.
\end{prop} 

\ssec{Proof of \propref{p:two maps}}

The proof of the proposition is \emph{not} a formal manipulation, as its statement involves the isomorphism
of \thmref{t:strange} for the two different parabolics, namely, $P$ and $P^-$. The corresponding geometric input is provided by 
\lemref{l:two CT} below.

\sssec{}

Let us identify 
$$\on{CT}^-_*\simeq \on{CT}_! \text{ and } \Eis_{\on{co},*}\simeq (\on{CT}^-_!)^\vee$$
via \thmref{t:weird adj}. 

\medskip

Then the map 
$$\psId_{\Bun_M,!}\circ \on{CT}_{\on{co},?} \to \on{CT}_!\circ \psId_{\Bun_G,!},$$
corresponding to \eqref{e:CT and Psi 2}, equals by definition the composition
\begin{multline*} %\label{e:comp 1}
\psId_{\Bun_M,!}\circ \on{CT}_{\on{co},?} \to \on{CT}_!\circ \Eis^-_!\circ \psId_{\Bun_M,!}\circ \on{CT}_{\on{co},?} 
\overset{\text{\eqref{e:weird adj -}}}
\simeq \\
\simeq \on{CT}_!\circ \psId_{\Bun_G,!}\circ  \Eis_{\on{co},*} \circ \on{CT}_{\on{co},?} \simeq
\on{CT}_!\circ \psId_{\Bun_G,!}\circ (\on{CT}^-_!)^\vee\circ (\Eis_!)^\vee = \\
= \on{CT}_!\circ \psId_{\Bun_G,!}\circ (\Eis_!\circ \on{CT}^-_!)^\vee \to \on{CT}_!\circ \psId_{\Bun_G,!},
\end{multline*}
where the first arrows comes from the unit of the $(\Eis^-_!,\on{CT}_!)$-adjunction, and the last arrow comes from 
the co-unit of the $(\Eis_!,\on{CT}^-_!)$-adjunction. 

\medskip

This corresponds to the following map of objects in $\Dmod(\Bun_G\times \Bun_M)$:
\begin{multline} \label{e:comp 2}
(\Eis_!\otimes \on{Id}_{\Dmod(\Bun_M)})\circ (\Delta_{\Dmod(\Bun_M)})_! (k_{\Bun_M}) \to \\
\to (\Eis_!\otimes (\on{CT}_!\circ \Eis^-_!))\circ (\Delta_{\Dmod(\Bun_M)})_! (k_{\Bun_M})=\\
=(\Eis_!\otimes \on{CT}_!)\circ (\on{Id}_{\Dmod(\Bun_M)}\otimes \Eis^-_!) \circ (\Delta_{\Dmod(\Bun_M)})_! (k_{\Bun_M})\simeq \\
\simeq (\Eis_!\otimes \on{CT}_!)\circ (\on{CT}^-_!\circ \on{Id}_{\Dmod(\Bun_G)}) \circ (\Delta_{\Dmod(\Bun_G)})_! (k_{\Bun_G})= \\
=((\Eis_!\circ \on{CT}^-_!)\otimes \on{CT}_!) \circ (\Delta_{\Dmod(\Bun_G)})_! (k_{\Bun_G})\to \\
\to (\on{Id}_{\Dmod(\Bun_G)}\otimes \on{CT}_!) \circ (\Delta_{\Dmod(\Bun_G)})_! (k_{\Bun_G}),
\end{multline}
where the isomorphism between the 3rd and the 4th lines is
\begin{multline*} %\label{e:comp 2}
(\on{Id}_{\Dmod(\Bun_M)}\otimes \Eis^-_!) \circ (\Delta_{\Dmod(\Bun_M)})_! (k_{\Bun_M})\simeq \\
\simeq ((\sfq^-\times \sfp^-)\circ \Delta_{\Bun_{P^-}})_!(k_{\Bun_{P^-}})\simeq \\
\simeq (\on{CT}^-_!\circ \on{Id}_{\Dmod(\Bun_G)}) \circ (\Delta_{\Dmod(\Bun_G)})_! (k_{\Bun_G}),
\end{multline*}
used in the proof of \thmref{t:strange}. 

\medskip

The assertion of the proposition amounts to showing that the composed map in \eqref{e:comp 2} equals
\begin{multline*}
(\Eis_!\otimes \on{Id}_{\Dmod(\Bun_M)})\circ (\Delta_{\Dmod(\Bun_M)})_! (k_{\Bun_M}) \simeq \\
\simeq  ((\sfp\times \sfq)\circ \Delta_{\Bun_P})_!(k_{\Bun_P})\simeq \\
\simeq (\on{Id}_{\Dmod(\Bun_G)}\otimes \on{CT}_!) \circ (\Delta_{\Dmod(\Bun_G)})_! (k_{\Bun_G}).
\end{multline*}

\sssec{}

The geometric input is provided by the following assertion, proved at the end of this subsection:

\begin{lem} \label{l:two CT}
The following diagram commutes:
$$
\CD
(\Delta_M)_! (k_M) @>>>  (\on{Id}_M\otimes (\on{CT}_!\circ \Eis^-_!))\circ (\Delta_M)_! (k_M) \\
@VVV   @VV{\sim}V   \\
((\on{CT}^-_!\circ \Eis_!)\otimes \on{Id}_M) \circ (\Delta_M)_! (k_M) & & 
(\on{Id}_M\otimes \on{CT}_!)\circ (\on{Id}_M\otimes \Eis^-_!) \circ (\Delta_M)_! (k_M) \\
@V{\sim}VV    @VV{\sim}V  \\
(\on{CT}^-_!\otimes \on{Id}_M)\circ (\Eis_!\otimes \on{Id}_M) \circ (\Delta_M)_! (k_M) & & 
(\on{Id}_M\otimes \on{CT}_!)\circ (\on{CT}^-_!\otimes \on{Id}_G) \circ (\Delta_G)_! (k_G) \\
@V{\sim}VV   @VV{\sim}V  \\
(\on{CT}^-_!\otimes \on{Id}_M)\circ (\on{Id}_G\otimes \on{CT}_!) \circ (\Delta_G)_! (k_G)
@>{\sim}>>  (\on{CT}^-_! \otimes \on{CT}_!) \circ (\Delta_G)_! (k_G)
\endCD
$$ 
where we use short-hand $\on{Id}_M,\Delta_M,k_M$ for $\on{Id}_{\Dmod(\Bun_M)}$, $\Delta_{\Bun_M}$ and $k_{\Bun_M}$,
respectively, and similarly for $G$. 
\end{lem}

Using the lemma, we rewrite the map in \eqref{e:comp 2} as follows:
\begin{multline*} %\label{e:comp 3}
(\Eis_!\otimes \on{Id}_{\Dmod(\Bun_M)})\circ (\Delta_{\Dmod(\Bun_M)})_! (k_{\Bun_M}) \to \\ 
\to
((\Eis_!\circ \on{CT}^-_!\circ \Eis_!)\otimes  \on{Id}_{\Dmod(\Bun_M)})\circ (\Delta_{\Dmod(\Bun_M)})_! (k_{\Bun_M}) = \\
=
((\Eis_!\circ \on{CT}^-_!)\otimes \on{Id}_{\Dmod(\Bun_M)})\circ (\Eis_! \otimes  \on{Id}_{\Dmod(\Bun_M)})\circ (\Delta_{\Dmod(\Bun_M)})_! (k_{\Bun_M}) \simeq \\
\simeq 
((\Eis_!\circ \on{CT}^-_!)\otimes \on{Id}_{\Dmod(\Bun_M)})\circ (\on{Id}_{\Dmod(\Bun_G)}\otimes \on{CT}_!) \circ 
(\Delta_{\Dmod(\Bun_G)})_! (k_{\Bun_G})\to \\
\to (\on{Id}_{\Dmod(\Bun_G)}\otimes \on{CT}_!) \circ 
(\Delta_{\Dmod(\Bun_G)})_! (k_{\Bun_G}),
\end{multline*}
and further as
\begin{multline*} %\label{e:comp 3}
(\Eis_!\otimes \on{Id}_{\Dmod(\Bun_M)})\circ (\Delta_{\Dmod(\Bun_M)})_! (k_{\Bun_M}) \to \\ 
\to
((\Eis_!\circ \on{CT}^-_!\circ \Eis_!)\otimes  \on{Id}_{\Dmod(\Bun_M)})\circ (\Delta_{\Dmod(\Bun_M)})_! (k_{\Bun_M}) = \\
=
((\Eis_!\circ \on{CT}^-_!)\otimes \on{Id}_{\Dmod(\Bun_M)})\circ (\Eis_! \otimes  \on{Id}_{\Dmod(\Bun_M)})\circ (\Delta_{\Dmod(\Bun_M)})_! (k_{\Bun_M}) \to \\
\to 
(\Eis_! \otimes  \on{Id}_{\Dmod(\Bun_M)})\circ (\Delta_{\Dmod(\Bun_M)})_! (k_{\Bun_M}) \simeq  \\
\simeq (\on{Id}_{\Dmod(\Bun_G)}\otimes \on{CT}_!) \circ 
(\Delta_{\Dmod(\Bun_G)})_! (k_{\Bun_G}).
\end{multline*}

However, the composition
\begin{multline*} 
(\Eis_!\otimes \on{Id}_{\Dmod(\Bun_M)})\circ (\Delta_{\Dmod(\Bun_M)})_! (k_{\Bun_M}) \to \\ 
\to
((\Eis_!\circ \on{CT}^-_!\circ \Eis_!)\otimes  \on{Id}_{\Dmod(\Bun_M)})\circ (\Delta_{\Dmod(\Bun_M)})_! (k_{\Bun_M}) = \\
=
((\Eis_!\circ \on{CT}^-_!)\otimes \on{Id}_{\Dmod(\Bun_M)})\circ (\Eis_! \otimes  \on{Id}_{\Dmod(\Bun_M)})\circ (\Delta_{\Dmod(\Bun_M)})_! (k_{\Bun_M}) \to \\
\to 
(\Eis_! \otimes  \on{Id}_{\Dmod(\Bun_M)})\circ (\Delta_{\Dmod(\Bun_M)})_! (k_{\Bun_M})
\end{multline*}
is the identity map, as it is induced by the map
$$\Eis_!\to \Eis_!\circ \on{CT}^-_!\circ \Eis_!\to \Eis_!,$$
comprised by the unit and co-unit of the $(\Eis_!,\on{CT}^-_!)$-adjunction, and the assertion follows. 

\sssec{Proof of \lemref{l:two CT}}

Let us recall from \cite[Sect. 1.3.2]{DrGa3} that the unit for the $(\Eis_!,\on{CT}^-_!)$ can be described as follows. The functor
$$\on{CT}_!\circ \Eis^-_!:\Dmod(\Bun_M)\to \Dmod(\Bun_M)$$
is given by
$$(\sfq)_!\circ (\sfp)^* \circ (\sfp^-)_!\circ (\sfq^-)^*,$$
which by base change along the diagram
$$
\xy
(-20,0)*+{\Bun_G}="X";
(20,0)*+{\Bun_M,}="Y";
(0,20)*+{\Bun_P}="Z";
(-40,20)*+{\Bun_{P^-}}="W";
(-60,0)*+{\Bun_M}="U";
(-20,40)*+{\Bun_{P^-}\underset{\Bun_G}\times \Bun_P}="V";
(-20,70)*+{\Bun_M}="T";
{\ar@{->}^{\sfp} "Z";"X"};
{\ar@{->}_{\sfq} "Z";"Y"};
{\ar@{->}_{\sfp^-} "W";"X"};
{\ar@{->}^{\sfq^-} "W";"U"}; 
{\ar@{->}^{'\sfp^-} "V";"Z"};
{\ar@{->}_{'\sfp} "V";"W"};
{\ar@{->}_{\bj} "T";"V"};
{\ar@{->}^{\on{id}} "T";"Y"};
{\ar@{->}_{\on{id}} "T";"U"};
\endxy
$$
can be rewritten as
$$(\sfq)_!\circ ({}'\sfp^-)_! \circ ({}'\sfp)^*\circ (\sfq^-)^*.$$

The natural transformation
$$\on{Id}_{\Dmod(\Bun_M)}\to \on{CT}_!\circ \Eis^-_!$$
is given by
$$(\on{id}_{\Bun_M})_!\circ (\on{id}_{\Bun_M})^*=
(\sfq)_!\circ  ({}'\sfp^-)_! \circ \bj_!\circ \bj^* \circ ({}'\sfp)^*\circ (\sfq^-)^*\to
(\sfq)_!\circ ({}'\sfp^-)_! \circ ({}'\sfp)^*\circ (\sfq^-)^*,$$
where the second arrow comes from the $(\bj_!,\bj^*)$-adjunction.

\medskip

The natural transformation
$$\on{Id}_{\Dmod(\Bun_M)}\to \on{CT}^-_!\circ \Eis_!$$
is described similarly, with the roles of $P$ and $P^-$ swapped. 

\medskip

Base change along
$$
\CD
\Bun_{P^-}\underset{\Bun_G}\times \Bun_P @>>>  \Bun_{P^-}\times \Bun_P  @>{\sfq^-\times \sfq}>> \Bun_M\times \Bun_M  \\
@VVV   @V{\sfp^-\times \sfp}VV  \\ 
\Bun_G   @>{\Delta_{\Bun_G}}>> \Bun_G\times \Bun_G
\endCD
$$ 
implies that the object
$$(\on{CT}^-_! \otimes \on{CT}_!) \circ (\Delta_{\Bun_G})_! (k_{\Bun_G})\in \Dmod(\Bun_M\times \Bun_M)$$
identifies with 
$$(\sfq^-\underset{\Bun_G}\times \sfq)_!(k_{\Bun_{P^-}\underset{\Bun_G}\times \Bun_P}),$$
where $\sfq^-\underset{\Bun_G}\times \sfq$ denotes the map
$$\Bun_{P^-}\underset{\Bun_G}\times \Bun_P \to \Bun_{P^-}\times \Bun_P  \overset{\sfq^-\times \sfq}\longrightarrow \Bun_M\times \Bun_M.$$

Now, the above description of the unit of the adjunctions implies that both circuits in the diagram in \lemref{l:two CT} are equal to the map 
$$(\Delta_{\Bun_M})_! (k_{\Bun_M})\to (\sfq^-\underset{\Bun_G}\times \sfq)_!(k_{\Bun_{P^-}\underset{\Bun_G}\times \Bun_P}),$$
that corresponds to the open embedding
$$\Bun_M \overset{\bj}\hookrightarrow \Bun_{P^-}\underset{\Bun_G}\times \Bun_P.$$

\qed

\ssec{Proof of \thmref{t:duality}}

We are finally ready to prove \thmref{t:duality}.

\medskip

We proceed by induction on the semi-simple rank of $G$. The case of a torus follows immediately 
from \corref{c:equiv on cusp}. Hence, we will assume that the assertion holds for all proper Levi
subgroups of $G$.

\sssec{}

\thmref{t:strange}, together with the induction hypothesis, imply that the essential image of $\Dmod(\Bun_G)_{\on{co,Eis}}$ under
$\psId_{\Bun_G,!}$ generates $\Dmod(\Bun_G)_{\Eis}$. 

\medskip

\corref{c:equiv on cusp} implies that the essential image of $\Dmod(\Bun_G)_{\on{co,cusp}}$
under $\psId_{\Bun_G,!}$ generates (in fact, equals) $\Dmod(\Bun_G)_{\on{cusp}}$.

\medskip

Hence, it remains to show that $\psId_{\Bun_G,!}$ is fully faithful. 

\sssec{}

The fact that $\psId_{\Bun_G,!}$ 
induces an isomorphism 
\begin{equation} \label{e:Hom to check}
\Hom_{\Dmod(\Bun_G)_{\on{co}}}(\CF',\CF)\to 
\Hom_{\Dmod(\Bun_G)}(\psId_{\Bun_G,!}(\CF'),\psId_{\Bun_G,!}(\CF))
\end{equation}
for $\CF'\in \Dmod(\Bun_G)_{\on{co,cusp}}$ follows from \propref{p:Hom out of cusp}.

\medskip

Hence, it remains to show that \eqref{e:Hom to check} is an isomorphism for 
$\CF'\in \Dmod(\Bun_G)_{\on{co,Eis}}$. The latter amounts to showing that
the functor $\psId_{\Bun_G,!}$ induces an isomorphism
\begin{multline*} \label{e:Hom out of cusp}
\Hom_{\Dmod(\Bun_G)_{\on{co}}}(\Eis_{\on{co},*}(\CF_M),\CF)\to  \\
\to \Hom_{\Dmod(\Bun_G)}(\psId_{\Bun_G,!}\circ \Eis_{\on{co},*}(\CF_M),\psId_{\Bun_G,!}(\CF))
\end{multline*}
for $\CF_M\in \Dmod(\Bun_M)_{\on{co}}$ for a \emph{proper} parabolic $P$ with
Levi quotient $M$. 

\sssec{}

Note that for $\CF_M\in \Dmod(\Bun_M)_{\on{co}}$ and $\CF\in \Dmod(\Bun_G)_{\on{co}}$
we have a commutative diagram:
$$
\CD
\Hom(\Eis_{\on{co},*}(\CF_M),\CF) @>>>  \Hom(\psId_{\Bun_G,!}\circ \Eis_{\on{co},*}(\CF_M),\psId_{\Bun_G,!}(\CF))  \\
& & @V{\text{\eqref{e:weird adj -}}}V{\sim}V \\
& &  \Hom(\Eis^-_!\circ \psId_{\Bun_M,!}(\CF_M),\psId_{\Bun_G,!}(\CF)) \\
@V{\sim}VV  @V{\sim}VV  \\
& &  \Hom(\psId_{\Bun_M,!}(\CF_M),\on{CT}^-_*\circ \psId_{\Bun_G,!}(\CF)) \\
& & @A{\text{\eqref{e:CT and Psi 2}}}AA   \\
\Hom(\CF_M,\on{CT}_{\on{co},?}(\CF)) @>>>  \Hom(\psId_{\Bun_M,!}(\CF_M),\psId_{\Bun_M,!}\circ \on{CT}_{\on{co},?}(\CF)). 
\endCD
$$

\medskip

The bottom horizontal arrow in the above diagram is an isomorphism by the induction hypothesis. 
Now, \propref{p:two maps} implies that the lower right vertical arrow is also
an isomorphism. 

\medskip

Hence, the upper horizontal arrow is also an isomorphism, as required.

\end{document}